\documentclass[a4paper,english]{article}
\pdfoutput=1
\usepackage[T1]{fontenc}
\usepackage[utf8]{inputenc}
\usepackage[english]{babel}
\usepackage{amsthm}
\usepackage{amsmath}
\usepackage{amssymb}
\usepackage{graphicx}
\usepackage{esint}
\usepackage[unicode=true,pdfusetitle, bookmarks=true,bookmarksnumbered=false,bookmarksopen=false,
 breaklinks=false,pdfborder={0 0 0},backref=false,colorlinks=false]
 {hyperref}
\hypersetup{pdftitle={Bernstein type inequalities for rational functions on analytic curves and arcs},
 pdfauthor={Sergei Kalmykov and Béla Nagy}}

\makeatletter
\usepackage{enumitem}		
\theoremstyle{plain}
\newtheorem{thm}{\protect\theoremname}
  \theoremstyle{plain}
  \newtheorem*{thm*}{\protect\theoremname}
\theoremstyle{plain}
\newtheorem{lem}[thm]{\protect\lemmaname}
\theoremstyle{plain}
\newtheorem{prop}[thm]{\protect\propositionname}

\makeatother

\usepackage{babel}
\providecommand{\lemmaname}{Lemma}
\providecommand{\theoremname}{Theorem}
\providecommand{\propositionname}{Proposition}

\newcommand*{\bC}{\mathbf{C}}
\newcommand*{\bCi}{\mathbf{C}_{\infty}}

\newcommand*{\bD}{\mathbb{D}}

\newcommand*{\dist}{\mathrm{dist}}

\newcommand*{\inter}{\mathrm{Int}}

\newcommand*{\constgg}{C_{6}} 
\newcommand*{\cfour}{C_{7}} 
\newcommand*{\cfive}{C_{8}} 
\newcommand*{\csix}{C_{10}} 
\newcommand*{\cthirteen}{C_{4}} 
\newcommand*{\cfourteen}{C_{5}} 
\newcommand*{\cnine}{C_{13}} 
\newcommand*{\csixteen}{C_{11}} 
\newcommand*{\cseventeen}{C_{12}} 
\newcommand*{\cnineteen}{C_{9}} 
\newcommand*{\ctwentyone}{C_{16}} 
\newcommand*{\ctwentytwo}{C_{17}} 

\newcommand*{\crbone}{C_{1}} 
\newcommand*{\crbtwo}{C_{2}} 
\newcommand*{\crbthree}{C_{3}} 
\newcommand*{\csharpone}{C_{19}} 
\newcommand*{\csharptwo}{C_{21}} 
\newcommand*{\csharpthree}{C_{20}} 

\newcommand*{\cdeltathree}{C_{18}} 

\newcommand*{\clogh}{C_{14}} 

\date{}

\begin{document}

\title{Bernstein type inequalities for rational functions on analytic curves and arcs}
\author{Sergei Kalmykov and Béla Nagy}

\maketitle


\begin{abstract}
Borwein and Erdélyi proved a Bernstein type inequality for rational functions on the unit circle and on the real line.
Here we establish asymptotically sharp extensions of their
inequalities for rational functions  on analytic Jordan arcs
and curves.
In the proofs key roles are played by Borwein-Erdélyi inequality on the unit circle, Gonchar-Grigorjan type
estimate of the norm of holomorphic part of meromorphic functions and
Totik’s construction of fast decreasing polynomials.


Classification (MSC 2010): 41A17, 30C20, 30E10
\end{abstract}

\section{Introduction}
Inequalities for polynomials and rational functions have rich history and numerous applications in different branches of mathematics, in particular in approximation theory (see, for example, \cite{BorweinErdelyi}, \cite{MMR}  and references therein).
One of the best known inequalities is Bernstein inequality for the derivative of polynomial on the interval $[-1,1]$. 
This inequality was generalized and improved in several directions, for instance, 
instead of one interval it was considered on compact subsets of the real line (see \cite{Totik2001}, \cite{Baran}) and on circles (\cite{NT2013},\cite{NT2014}), on lemniscates \cite{N2005}, 
on systems of Jordan curves \cite{NT2005}, and more recently, on analytic Jordan arcs \cite{JMAA}.
For rational functions, Bernstein type inequalities were shown 
for a circle \cite{BorweinErdelyipaper}, on compact subsets of the real line and  circle (see \cite{BorweinErdelyipaper}, \cite{Lukashov2004}, \cite{DubininKalmykov}). 
For improvements of some of such  inequalities, see \cite{Dubinin} and references therein.  

The aim of this paper is to extend  the mentioned above Borwein-Erd\'{e}lyi inequality for rational functions from a circle to an arbitrary analytic Jordan curve when poles are from a given compact set away from the curve. 
The obtained inequality is asymptotically sharp. 
As a consequence we get a similar estimate on an analytic arc as well. 
All the results are formulated in terms of normal derivatives of Green functions of corresponding domains. 
For the necessary background on potential theory, we refer to 
\cite{Ransford} and \cite{SaffTotik}.
The basic theme of this paper is similar to a
recently developed method in \cite{JMAA}, but most of the
details are rather different.
This approach is based on using the Borwein-Erd\'{e}lyi inequality as a model case, estimates for Green's functions, fast decreasing polynomials, Gonchar-Grigorjan
estimate of the norm of holomorphic part of meromorphic functions, and an appropriate interpolation; for the case of an analytic arc we apply  ``open-up''   mapping (a rational function).
Although Bernstein
inequality was developed to
prove inverse theorems in approximation theory, we use direct
approximation theorems to establish Theorem \ref{thm:analrational}.

The structure of the paper is the following: in Section \ref{sec:statement} we formulate the statements of new results, in Section \ref{sec:proofone} auxiliary facts are collected and the main theorem is proved, in Section 4 we consider the case of one analytic arc, the last section is devoted to sharpness.

\section{Statements of the new results}
\label{sec:statement}

We denote the unit disk by $\bD=\{z:|z|<1\}$,
$\bD^*=\{z:|z|>1\}\cup\{\infty\}$ is called exterior of the unit disk and
$\bCi=\bC\cup\{\infty\}$ denotes the extended complex plane.
We frequently use $g_D(z,\alpha)$ for Green's function of domain $D$ with pole at $\alpha\in D$.

\begin{thm}
\label{thm:analrational}
Let $\Gamma\subset\bC$ be an analytic Jordan curve, $u_{0}\in\Gamma$. 
Let $G_1\subset \bC$ be the interior of $\Gamma$, $G_2$ be the exterior domain $G_2:=\bC_\infty \setminus \left( \Gamma \cup G_1\right)$
and let $Z\subset G_1 \cup G_2$ be a closed set. 
Denote the two normals to $\Gamma$ at $u_{0}$ by $n_{1}\left(u_{0}\right)$
and $n_{2}\left(u_{0}\right)$, $n_{2}\left(u_{0}\right)=-n_{1}\left(u_{0}\right)$, where
$n_1\left(u_{0}\right)$ and $n_2\left(u_{0}\right)$ are pointing inward and outward respectively.

Then, for any rational function $f\left(u\right)$ with poles
in $Z$ only, we have
\begin{multline*}
\left|f'\left(u_{0}\right)\right|
\le
\left(1+o\left(1\right)\right)\left\Vert f\right\Vert_{\Gamma}\\
\cdot\max\left(
\sum_{\alpha}
\frac{\partial}{\partial n_{1}\left(u_{0}\right)}
g_{G_1}\left(u_{0},\alpha\right),\sum_{\beta}
\frac{\partial}{\partial n_{2}\left(u_{0}\right)}
g_{G_2}\left(u_{0},\beta\right)\right)
\end{multline*}
where the sum with $\alpha$ (or $\beta$) is taken over all poles of $f$ in $G_1$ (or in $G_2$, respectively), counting multiplicities,
and $o\left(1\right)$ denotes an error term that depends on $\Gamma$,
$u_{0}$ and $Z$, tends to $0$ as the total degree of $f$ tends
to infinity and is independent of $f$ itself.
 \end{thm}

Applying an appropriate open-up mapping, we obtain

\begin{thm}
\label{thm:analrationalarc}
Let $\Gamma_0\subset\bC$ be an analytic Jordan arc, $z_{0}\in\Gamma_0$ not endpoint. 
Denote the two normals to $\Gamma_0$ at $z_{0}$ by $n_{1}\left(z_{0}\right)$
and $n_{2}\left(z_{0}\right)$, $n_{2}\left(z_{0}\right)=-n_{1}\left(z_{0}\right)$.
Let $G:=\bC_\infty \setminus \Gamma_0$ be 
the complementing domain
and let $Z\subset G$ be a closed
set. 
Then, for any rational function $f\left(z\right)$ with poles
in $Z$ only, we have
\begin{multline*}
\left|f'\left(z_{0}\right)\right|
\le
\left(1+o\left(1\right)\right)
\left\Vert f\right\Vert_{\Gamma_0}
\\
\cdot\max\left(
\sum_{\beta}
\frac{\partial}{\partial n_{1}\left(z_{0}\right)}
g_{G}\left(z_{0},\beta\right),
\sum_{\beta}
\frac{\partial}{\partial n_{2}\left(z_{0}\right)}
g_{G}\left(z_{0},\beta\right)
\right)
\end{multline*}
where the sum with $\beta$ is taken over all poles of $f$ in $G$ counting multiplicities,
and $o\left(1\right)$ denotes an error term that depends on $\Gamma_0$,
$z_{0}$ and $Z$, tends to $0$ as the total degree of $f$ tends
to infinity and is independent of $f$ itself.
\end{thm}

Theorem \ref{thm:analrational} is asymptotically sharp as the following
theorem shows.
\begin{thm}
\label{thm:analratsharp} 
We use the notations of Theorem \ref{thm:analrational}.

Then
there exists a sequence of rational functions $\left\{f_{n}\right\}$ with $\deg\left(f_{n}\right)=n\rightarrow\infty$
with poles in $Z$
such that
\begin{multline*}
\left|f_{n}'\left(u_{0}\right)\right|
\ge 
\left(1-o\left(1\right)\right)
\left\Vert f_{n}\right\Vert_{\Gamma}
\\ \cdot
\max\left(
\sum_{\alpha}
\frac{\partial}{\partial n_{1}\left(u_{0}\right)}
g_{G_1}\left(u_{0},\alpha\right),\sum_{\beta}
\frac{\partial}{\partial n_{2}\left(u_{0}\right)}
g_{G_2}\left(u_{0},\beta\right)
\right)
\end{multline*}
where $o(1)$ tends to $0$ as $\deg(f_n)=n\rightarrow\infty$ and it depends on $\Gamma$, $u_0$ and $Z$ too
and the sum with $\alpha$ (or $\beta$) is taken over all poles of $f$ in $G_1$ (or in $G_2$, respectively), counting multiplicities.
\end{thm}


A particular sequence of rational functions showing sharpness in Theorem \ref{thm:analrationalarc} can be obtained from Theorem 2 from \cite{JMAA} in standard way if we take a point from $Z$ and apply a fractional-linear mapping which maps this point to infinity.

The error term $o(1)$ in Theorems \ref{thm:analrational} and \ref{thm:analrationalarc} cannot be dropped in general,
even for polynomials, see \cite{N2005}.

\section{Some background results and the proof of Theorem \ref{thm:analrational} }
\label{sec:proofone}

\subsection{A ``rough'' Bernstein type inequality}

We  need the following ``rough'' Bernstein type inequality on
Jordan curves.
\begin{prop}
\label{prop:rough_berns}
Let $\Gamma$ be a $C^{2}$ smooth Jordan curve and $Z\subset\bCi\setminus\Gamma$
be a closed set. 
Then, there exists $\crbone >0$ such that for any
$u_{0}\in\Gamma$ and any rational function $f$ with poles in $Z$
only, we have
\[
\left|f'\left(u_{0}\right)\right|
\le 
\crbone \deg\left(f\right)\left\Vert f\right\Vert_{\Gamma}.
\]
\end{prop}
\begin{proof}
This quickly follows from  an idea attributed to Szegő (Cauchy
integral formula around the point, see e.g. \cite{Widom}, p. 133). 
Fix $u_{0}\in\Gamma$ and 
and put 
$G_1:=\mathrm{Int} \Gamma$,
$G_2:=\bCi \setminus \left( \Gamma \cup \mathrm{Int}\Gamma\right)$
and let
$\delta_{0}:=\dist\left(Z,\Gamma\right)$.
Now fix $\alpha_{j}^{\left(0\right)}\in G_{j}$ arbitrarily,
$j=1,2$. 
Let $N_{1}$ be the total order of poles of $f$ in $G_{1}$, $N_{2}$
be the total order of poles of $f$ in $G_{2}$. 
Obviously, $N_{1}+N_{2}=\deg\left(f\right)$.

Since the poles are from a compact set,
it is standard (see, e.g. \cite{multiplygg}, Lemma 1)  that there exist
$\delta_{0,1}>0$ and $\crbtwo>0$ such that 
if $\dist\left(u,\Gamma\right)<\delta_{0,1}$, $u\in G_j$ and
$\alpha_{1},\alpha_{2}\in Z\cap G_{j}$, $j=1,2$, then 
\[
\frac{1}{\crbtwo}
\le
\frac{g_{G_{j}}\left(u,\alpha_{1}\right)}{g_{G_{j}}\left(u,\alpha_{2}\right)}
\le \crbtwo.
\]

Let $\gamma:=\left\{ u:\left|u-u_{0}\right|=1/\deg\left(f\right)\right\} $
(assume $\deg(f)>2/\delta_{0}$) and we use  Bernstein-Walsh estimate
for $f$ on $G_{1}$ and $G_{2}$ 
(see, e.g. \cite{Ransford}, Theorem 5.5.7, p. 156):
\[
\left|f\left(u\right)\right|\le\left\Vert f\right\Vert_{\Gamma}
\exp\left(\sum_{\alpha}g_{G_{j}}\left(u,\alpha\right)\right)
\]
where the sum is taken over all poles
$\alpha$ of $f$ in $G_{j}$,
each of them appearing as many times as the order of the pole of $f$
at $\alpha$ and $u\in G_j$. 
It is again standard (see, e.g., Lemma 1 from \cite{multiplygg}) 
that there exist $\delta_{0,2}>0$ and $\crbthree>0$ such that if $\dist\left(u,\Gamma\right)<\delta_{0,2}$,
then $g_{G_{j}}\left(u,\alpha_{j}^{\left(0\right)}\right)
\le
\crbthree\dist\left(u,\Gamma\right)\le \crbthree\left|u-u_{0}\right|$.
This way we can estimate the sum in the previous displayed formula
as follows
\begin{multline*}
\sum_{\alpha}g_{G_{j}}\left(u,\alpha\right)
\le
\sum_{\alpha} \crbtwo
g_{G_{j}}\left(u,\alpha_{j}^{\left(0\right)}\right)
\le
\sum_{\alpha} \crbtwo \crbthree \left|u-u_{0}\right|
\\ =
\crbtwo \crbthree N_{j}\left|u-u_{0}\right|
=
\crbtwo \crbthree \frac{N_{j}}{\deg\left(f\right)}\le
\crbtwo \crbthree.
\end{multline*}

We apply Cauchy integral  formula
\begin{multline*}
\left|f'\left(u_{0}\right)\right|
=
\left|\frac{1}{2\pi i}
\int_{\gamma}\frac{f\left(u\right)}{\left(u_{0}-u\right)^{2}}du\right|
\le
\frac{1}{2\pi}\int_{\gamma}\left|\frac{f\left(u\right)}{\left(u_{0}-u\right)^{2}}\right|\left|du\right|
\\ \le
\frac{1}{2\pi}2\pi\frac{1}{\deg( f)}\frac{\left\Vert f\right\Vert_{\Gamma}\exp\left(\crbtwo \crbthree\right)}
{\left(\deg(f)\right)^{-2}}
=
\left\Vert f\right\Vert_{\Gamma}\deg(f)\exp\left(\crbtwo \crbthree \right).
\end{multline*}

The proposition is proved with 
$\crbone =\exp\left(\crbtwo \crbthree\right)$
which is independent of $f$, $\deg( f)$ and $u_{0}$.
\end{proof}

\subsection{Conformal mappings on simply connected domains}

Recall, for a given curve $\Gamma$, $G_1$ denotes the interior of $\Gamma$ and $G_2$ denotes the unbounded component of $\bC_\infty\setminus \Gamma$.

As earlier, $\bD=\left\{ v:\ \left|v\right|<1\right\} $
and $\bD^{*}=\left\{ v:\ \left|v\right|>1\right\} \cup\left\{ \infty\right\} $.
With these notations, $\partial G_{1}=\partial G_{2}$. Using Kellogg-Warschawski
theorem (see e.g. \cite{Pommerenke} p. 49, Theorem 3.6), if the boundary
is $C^{1,\alpha}$ smooth, then the Riemann mappings of $\bD,\bD^{*}$
onto $G_{1},G_{2}$ respectively and their derivatives can be extended
continuously to the boundary.

Under analyticity assumption, we can compare the Riemann mappings
as follows.
\begin{prop}
\label{prop:green_deriv_g_one_g_two}Let $u_{0}\in\partial G_{1}=\partial G_{2}$
be fixed. Then there exist two Riemann mappings $\Phi_{1}:\bD\rightarrow G_{1}$,
$\Phi_{2}:\bD^{*}\rightarrow G_{2}$ such that $\Phi_{j}\left(1\right)=u_{0}$
and $\left|\Phi_{j}'\left(1\right)\right|=1$, $j=1,2$. 

If $\partial G_{1}=\partial G_{2}$ is analytic, then there exist
$0\le r_{1}<1<r_{2}\le\infty$ such that $\Phi_{1}$ extends to $D_{1}:=\left\{ v:\ \left|v\right|<r_{2}\right\} $,
$G_{1}^{+}:=\Phi_{1}\left(D_{1}\right)$ and $\Phi_{1}:D_{1}\rightarrow G_{1}^{+}$
is a conformal bijection, and similarly, $\Phi_{2}$ extends to $D_{2}:=\left\{ v:\ \left|v\right|>r_{1}\right\} \cup\left\{ \infty\right\} $,
$G_{2}^{+}:=\Phi_{2}\left(D_{2}\right)$ and $\Phi_{2}:D_{2}\rightarrow G_{2}^{+}$
is a conformal bijection. \end{prop}

This proposition is Proposition 7 in \cite{JMAA}.

From now on, we fix such two conformal mappings. These mappings and domains are depicted on figure \ref{fig:analext}.
\begin{figure}
\begin{center}
\includegraphics[keepaspectratio,width=\textwidth]{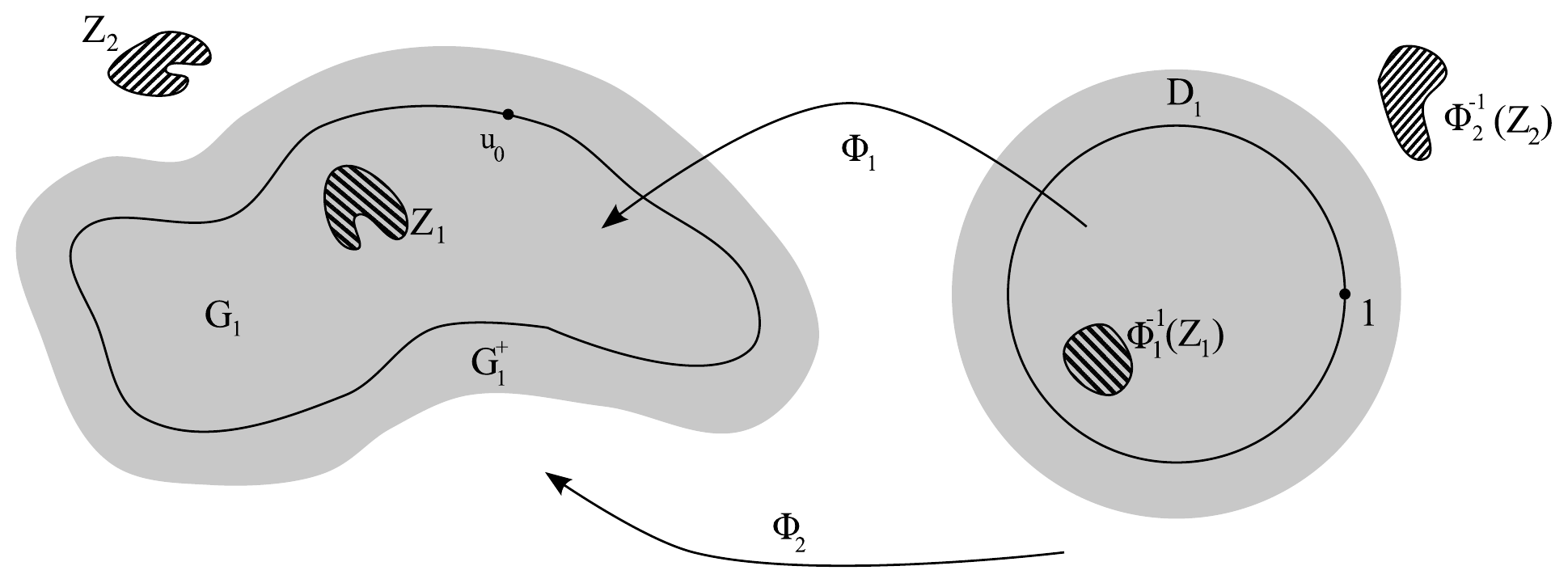}
\end{center}
\caption{The two conformal mappings $\Phi_1$, $\Phi_2$, the domain $D_1$ and the possible location of poles}
\label{fig:analext}
\end{figure}
We may assume that $D_1$ and $Z_2$ are of positive distance from one another (by slightly decreasing $r_1$, if necessary).

Denote the normal vector (of unit length) to $\Gamma$ at $u_0$ 
or $\partial\bD$ at $1$
pointing inward
by $n_1\left(u_0\right)$ or
$n_1(1)$ respectively.
Similarly, the outward normal vectors are denoted by $n_2\left(u_0\right)$
and $n_2(1)$.

\begin{prop}
\label{prop:green_deriv_unitdisk}
The following hold for arbitrary $\alpha\in G_1$, $\beta \in G_2$
with $\alpha':=\Phi_1^{-1}(\alpha)$, $\beta':=\Phi_2^{-1}(\beta)$
if $\beta'\ne\infty$
\begin{gather*}
\frac{\partial}{\partial n_{1}\left(u_{0}\right)}
g_{G_{1}}\left(u_{0},\alpha\right)
=
\frac{\partial}{\partial n_{1}\left(1\right)}
g_{\bD}\left(1,\alpha'\right)
=
\frac{1-\left|\alpha'\right|^{2}}{\left|1-\alpha'\right|^{2}},
\\
\frac{\partial}{\partial n_{2}\left(u_{0}\right)}
g_{G_{2}}\left(u_{0},\beta\right)
=
\frac{\partial}{\partial n_{2}\left(1\right)}
g_{\bD^{*}}\left(1,\beta'\right)
=
\frac{\left|\beta'\right|^{2}-1}{\left|1-\beta'\right|^{2}},
\end{gather*}
and if $\beta'=\infty$, then
\[
\frac{\partial}{\partial n_{2}\left(u_{0}\right)}
g_{G_{2}}\left(u_{0},\beta\right)
=
\frac{\partial}{\partial n_{2}\left(1\right)}
g_{\bD^{*}}\left(1,\infty\right)=1.
\]
\end{prop}

This proposition is a slight generalization of Proposition 8 from \cite{JMAA} with the same proof.

\subsection{Fast decreasing rational functions with prescribed poles}

The next result is based on a general construction of fast
decreasing polynomials by Totik, see \cite{Totik}, Corollary
4.2 and Theorem 4.1 too.


\begin{thm}
\label{thm:fast_decr}Let $\tilde{K}\subset\bC$ be a compact set,
$\tilde{u}\in\partial\tilde{K}$ be a boundary point. Assume that
$\tilde{K}$ satisfies the touching-outer-disk condition, that is,
there exists a closed disk (with positive radius) such that its intersection
with $\tilde{K}$ is $\left\{ \tilde{u}\right\} $. Let $Z_{2}^{*}\subset\bCi\setminus\tilde{K}$
be a finite set and $\tau>1$.

Then there exist $\cthirteen,\cfourteen>0$ with the following properties.
For any given multiplicity function $m:Z_{2}^{*}\rightarrow\left\{ 1,2,\ldots\right\} $
introduce $\tilde{n}:=\sum_{\alpha}m\left(\alpha\right)$ where the sum is
taken for $\alpha\in Z_{2}^{*}$ and there exists a rational function
$Q$ such that $Q\left(\tilde{u}\right)=1$, $\left\Vert Q\right\Vert _{\tilde{K}}\le1$,
$Q$ has poles at the points of $Z_{2}^{*}$ only and the order of
the pole of $Q$ at $\alpha\in Z_{2}^{*}$ is at most $m\left(\alpha\right)$
and for $u\in\tilde{K}$ we have
\[
\left|Q\left(u\right)\right|\le\cthirteen\exp\left(-\cfourteen \tilde{n}\left|u-\tilde{u}\right|^{\tau}\right).
\]
\end{thm}
\begin{proof}
Roughly speaking, we transform for each $\alpha\in Z_{2}^{*}$ the
fast decreasing polynomials to move their poles from $\infty$ to
$\alpha$ and multiply them together. 

Let $\alpha\in Z_{2}^{*}$ be fixed and let $\tilde{K}_{\alpha}:=\left\{ \frac{1}{u-\alpha}:u\in\tilde{K}\right\} $.
Obviously, $\tilde{K}_{\alpha}$ is a compact set and also satisfies
the touching-outer-disk condition at $\frac{1}{\tilde{u}-\alpha}$.
Using theorem 4.1 on page 2065 from \cite{Totik}, we 
know that there are $d_{\alpha,\tau}, D_{\alpha,\tau}>0$
depending on $\alpha$, $\tilde{K}$ and $\tau$ only
and we get a polynomial
$Q_{\alpha}$ such that $Q_{\alpha}\left(\frac{1}{\tilde{u}-\alpha}\right)=1$,
$\left\Vert Q_{\alpha}\right\Vert _{\tilde{K}_{\alpha}}=1$, $\deg\left(Q_{\alpha}\right)\le m\left(\alpha\right)$
and for $u\in\tilde{K}$
\begin{multline*}
\left|Q_{\alpha}\left(\frac{1}{u-\alpha}\right)\right|\le D_{\alpha,\tau}\exp\left(-d_{\alpha,\tau}\, m\left(\alpha\right)\,\left|\frac{1}{u-\alpha}-\frac{1}{\tilde{u}-\alpha}\right|^{\tau}\right)\\
\le D_{\alpha,\tau}\exp\left(-d_{\alpha,\tau}\, m\left(\alpha\right)\,\dist\left(\partial\tilde{K},Z_{2}^{*}\right)^{-2\tau}\left|u-\tilde{u}\right|^{\tau}\right).
\end{multline*}

Let $Q\left(u\right):=\prod_{\alpha\in Z_{2}^{*}}Q_{\alpha}\left(\frac{1}{u-\alpha}\right)$.
We immediately have $Q\left(\tilde{u}\right)=1$,
$\left\Vert Q\right\Vert _{\tilde{K}}=1$,
and $Q$ has poles at the points of $Z_{2}^{*}$ only, the order of
the pole of $Q$ at $\alpha\in Z_{2}^{*}$ is at most $m\left(\alpha\right)$.
We multiply the estimates above together, so
\begin{multline*}
\left|Q\left(u\right)\right|\le\left(\prod_{\alpha}D_{\alpha,\tau}\right)\exp\left(-\dist\left(\partial\tilde{K},Z_{2}^{*}\right)^{-2\tau}\left|u-\tilde{u}\right|^{\tau}\sum_{\alpha}d_{\alpha,\tau}\, m\left(\alpha\right)\right)\\
\le\cthirteen\exp\left(-\left(\min_{\alpha}d_{\alpha,\tau}\right)\dist\left(\partial\tilde{K},Z_{2}^{*}\right)^{-2\tau}n\left|u-\tilde{u}\right|^{\tau}\right)\\
=\cthirteen\exp\left(-\cfourteen n\left|u-\tilde{u}\right|^{\tau}\right)
\end{multline*}
where $\cthirteen=\prod_{\alpha}D_{\alpha,\tau}$ and $\cfourteen=\left(\min_{\alpha}d_{\alpha,\tau}\right)\dist\left(\partial\tilde{K},Z_{2}^{*}\right)^{-2\tau}$.
\end{proof}

\subsection{Outer touching circles
and other quantities}
\label{sec:touchcircles}

In this subsection we construct 
some auxiliary sets and the fast 
decreasing rational functions.
We also use $\arg(z)$ in the following sense: if $z\ne 0$, then $\arg(z)=z/|z|$
and $\arg(0)=0$.

Let $\Gamma_{1}=\left\{ v:\left|v\right|=1+\delta_{1}\right\} $ 
where $0<\delta_1<r_2-1$, $\delta_1<1/2$ is fixed.
It is important that $\delta_{1}$ depends on $G_{1}$ only.

Let $D_{3}:=\left\{ v:\left|v-2\right|<1\right\} $, this disk
touches the unit disk at $1$.

For every $\delta_{2,3}\in\left(0,\delta_1\right]$,
$\left\{ v:\ \left|v\right|=1+\delta_{2,3}\right\} \cap\partial D_{3}$
consists of exactly two points, $v_{1}^{*}=v_{1}^{*}\left(\delta_{2,3}\right)$
and $v_{2}^{*}=v_{2}^{*}\left(\delta_{2,3}\right)$. 
It is easy to
see that the length of the two arcs of $\left\{ v:\left|v\right|=1+\delta_{2,3}\right\} $
lying between $v_{1}^{*}$ and $v_{2}^{*}$ are different, therefore,
by reindexing them, we can assume that the shorter arc is going from
$v_{1}^{*}$ to $v_{2}^{*}$  counterclockwise. 
Elementary geometric
considerations show that for all $v$, $1\le\left|v\right|\le1+\delta_{2,3}$
with $\arg v\in\left\{ \arg v_{j}^{*}\left(\delta_{2,3}\right):\ j=1,2\right\} $,
we have (since $\delta_{2,3}<1$) 
\begin{equation}
\frac{1}{2}\sqrt{\delta_{2,3}}
\le
\left|v-1\right|
\le
2\sqrt{\delta_{2,3}}.\label{eq:w_star_dist}
\end{equation}

Let 
\[
K_{v}^{*}:=\left\{ v:\ \left|v\right|
\le
1+\delta_{1}\right\} \setminus D_{3}.
\]
Obviously, this $K_{v}^{*}$ is a compact set and satisfies the touching-outer-disk
condition at $1$ of Theorem
\ref{thm:fast_decr}.

Consider 
\[
K_{u}^{*}:=
\Phi_{2}\left[K_{v}^{*}\cap\bD^{*}\right]
\cup\Phi_{1}\left[K_{v}^{*}\cap\bD^{*}\right]\cup G_{1}.
\]
This is a compact set and also satisfies the touching-outer-disk condition
at $u_{0}=\Phi_{2}\left(1\right)$ of Theorem \ref{thm:fast_decr}.
Obviously,
$\partial G_{2}\subset K_{u}^{*}$, $G_{1}\subset K_{u}^{*}$ , $u_{0}\in K_{u}^{*}$
and if $v\in K_{v}^{*}$, then
$\Phi_{1}\left(v\right)\in K_{u}^{*}$
and $\Phi_{2}\left(v\right)\in K_{u}^{*}$ too.

Take any finite set $Z_2^*\subset Z_2$
with arbitrary multiplicity function such that
the total multiplicity is at most $N_3:=n^{3/4}$.
Now
applying Theorem \ref{thm:fast_decr}, 
there exists a fast decreasing rational function
for $K_{u}^{*}$ at $u_{0}$
and its degree is $N_4$, where $N_4 \le N_3$.
Denote the poles of $Q$ by $\zeta_1, \ldots, \zeta_{N_4}\in Z_2^{*}$ counting multiplicities.

Then $\deg(Q)=N_4\le N_3$, $Q\left(u_{0}\right)=1$, $\left|Q\left(u\right)\right|\le 1$
on if $u\in K_{u}^{*}$, moreover
\begin{equation}
\left|Q\left(u\right)\right|\le \cthirteen\exp\left(-\cfourteen N_3 |u-u_0|^\tau\right).\label{eq:Q_fast_decreases}
\end{equation}
Note that  $\cthirteen$ and $\cfourteen$ depend on $Z_2^*$ and $K_u^*$ but they are independent of
the multiplicity function and $N_3$ and $n$.

For simplicity, we put $\tau:=4/3$.

\subsection{The proof of Theorem \ref{thm:analrational}}

\subsubsection{Decomposition of the rational function}

It is easy to decompose $f$ into sum of rational functions, that
is,
\[
f=f_{1}+f_{2}
\]
where $f_{1}$ is a rational function with poles in $Z_{1}$, $f_{1}\left(\infty\right)=0$
and $f_{2}$ is a rational function with poles in $Z_2$.
This decomposition is unique. 
Put $N_1:=\deg\left(f_1\right)$,
$N_2:=\deg\left(f_2\right)$, then $N_1+N_2=n$.
Denote the poles of $f_1$ by
$\alpha_{j}$, $j=1,\ldots,N_1$
and
the poles of $f_2$ by
$\beta_{j}$, $j=1,\ldots,N_2$
 (with counting the orders of the poles).

Now fix
\begin{equation*}
\delta_{2,1}:=\frac{1}{2n}, 
\quad \delta_{2,3}:=\min\left(n^{-2/3}, 
\delta_1 \right).
\end{equation*} 


We use a Gonchar-Grigorjan type estimate
for $f_{2}$ on $G_{1}$ (see Theorem 1 in \cite{multiplygg}) so there exists $\constgg=\constgg(G_1)>0$ such that we have
\begin{equation}
\left\Vert f_{2}\right\Vert_{\Gamma}
\le 
\constgg \left(G_{1}\right)\log\left(n\right)
\left\Vert f\right\Vert_{\Gamma}.
\label{eq:f_two_norm_est}
\end{equation}
Obviously, we have 
\begin{equation}
\left\Vert f_{1}\right\Vert_{\Gamma}
\le
\left(1+\constgg\left(G_{1}\right)\log\left(n\right)\right)
\left\Vert f\right\Vert_{\Gamma}.
\label{eq:f_one_norm_est}
\end{equation}

Consider 
\[
\varphi_{1}\left(v\right):=f_{1}\left(\Phi_{1}\left(v\right)\right).
\]
We know that 
\begin{equation}
\left\Vert \varphi_{1}\right\Vert _{\partial\bD}=\left\Vert f_{1}\right\Vert _{\partial G_{2}}\label{eq:phi_one_norm}
\end{equation}
 and $\left|\varphi_{1}'\left(1\right)\right|=\left|f_{1}'\left(u_{0}\right)\right|$. 

We use the fast decreasing rational function $Q$ from Subsection \ref{sec:touchcircles}.

We decompose ``the essential part of'' $\varphi_{1}$ as follows
\begin{equation}
Q\circ\Phi_{1}\cdot\varphi_{1}=\varphi_{1r}+\varphi_{1e}\label{eq:phi_decomp}
\end{equation}
where $\varphi_{1r}$ is a rational function, $\varphi_{1r}\left(\infty\right)=0$
and $\varphi_{1e}$ is holomorphic in $\bD$. 
We use the decomposition
\begin{equation*}
(Qf)\circ \Phi_1
=
\left(Q\left(f_1+f_2\right)\right)\circ \Phi_1 
= \varphi_{1r}+\varphi_{1e}+
\left(Q f_2\right)\circ \Phi_1.
\end{equation*}

We apply the Gonchar-Grigorjan type
estimate (see Theorem 1 \cite{multiplygg})  again for $\varphi_{1e}$ on
$\bD$, this way the following sup norm estimate holds
\begin{equation}
\left\Vert \varphi_{1e}\right\Vert _{\partial\bD}
\le 
\constgg\left(\bD\right)
\log\left(n\right)\left\Vert Q\circ\Phi_{1}\cdot\varphi_{1}\right\Vert_{\partial\bD}
\le 
\constgg\left(\bD\right)
\log\left(n\right)\left\Vert \varphi_{1}\right\Vert_{\partial\bD}
\label{eq:phi_one_e_norm_est}
\end{equation}
where $C_{1}\left(\bD\right)$ is a constant independent of $\varphi_{1}$.

Furthermore, we can estimate $\varphi_{1e}\left(v\right)$ on $v\in D_{1}\setminus\bD$
as follows
\begin{equation}
\left|\varphi_{1e}\left(v\right)\right|=\left|\left(Q\cdot f_{1}\right)\circ\Phi_{1}\left(v\right)-\varphi_{1r}\left(v\right)\right|\le\left|\left(Q\cdot f_{1}\right)\circ\Phi_{1}\left(v\right)\right|+\left|\varphi_{1r}\left(v\right)\right|.\label{est:phi_one_e_v}
\end{equation}

We also need to estimate $Q$ outside $G_1$  as
follows. 

Using  Bernstein-Walsh
estimate, we can write for $v\in D_{1}\setminus\bD$
\[
\left|Q\left(\Phi_{1}\left(v\right)\right)\right|
\le
1\cdot\exp\left(
\sum_{j=1}^{N_4}  
g_{G_{2}}\left(\Phi_{1}\left(v\right),\zeta_j \right)\right)
\]
where we use that the set $\Phi_{1}\left(D_{1}\setminus\bD\right)$ is bounded,
\[
\cfour:=\sup\left\{ g_{G_{2}}\left(\Phi_{1}\left(v\right),\beta\right):\ v\in D_{1}\setminus\bD,\,\beta\in Z_{2}\right\} <\infty.
\]
Therefore, for all $v\in D_{1}\setminus\bD$, 
\[
\left|\left(Q\cdot f_{1}\right)\circ\Phi_{1}\left(v\right)\right|
\le 
e^{\cfour N_4}\left\Vert f_{1}\right\Vert _{\Gamma}.
\]

This way we can continue (\ref{est:phi_one_e_v}) and we use $u=\Phi_{1}\left(v\right)$
and that $\varphi_{1r}$ is a rational function with no poles
outside $\bD$ and the maximum principle for $\varphi_{1r}$ 
\[
\le e^{\cfour N_4}\left|f_{1}\left(u\right)\right|+\left\Vert \varphi_{1r}\right\Vert _{\partial\bD}
\le 
e^{\cfour N_4}\left\Vert f_{1}\right\Vert _{\Gamma}+\left\Vert \varphi_{1}\right\Vert _{\partial\bD}+\left\Vert \varphi_{1e}\right\Vert _{\partial\bD}
\]
and here we used that $f_{1}$ has no pole in $G_{2}$ and the maximum
principle. We can estimate these three sup norms with the help of
(\ref{eq:f_one_norm_est}) and (\ref{eq:phi_one_norm}), (\ref{eq:f_one_norm_est})
and (\ref{eq:phi_one_e_norm_est}), (\ref{eq:phi_one_norm}), (\ref{eq:f_one_norm_est}).
Hence we have for $v\in D_{1}\setminus\bD$ 
\begin{multline}
\left|\varphi_{1e}\left(v\right)\right|
\le
\left(e^{\cfour N_4 }+1+
\constgg\left(\bD\right)
\log\left(n\right)\right)\left(1+
\constgg\left(G_{1}\right)
\log\left(n\right)\right)\left\Vert f\right\Vert _{\Gamma}
\\
=O\left(\log\left(n\right)e^{\cfour N_4}\right)\left\Vert f\right\Vert_{\Gamma}.
\label{eq:phi_one_e_v_est}
\end{multline}

\subsubsection{Approximating  the interior error function}

In this subsection we
construct an approximation to $\varphi_{1e}$ which is holomorphic on a larger set containing $\bD$.
Later we will use properties \eqref{eq:phi_one_e_sup_err_est} and \eqref{eq:phi_one_e_deriv_at_one} only.

The approximation is done by interpolating $\varphi_{1e}$ as follows. $\varphi_{1e}$
is holomorphic in $D_{1}=\left\{ v\in\bC:\,|v|<1+\delta_{1}\right\} $.

Put
\[
N_5:=N_2+N_4+[n^{4/5}],
\]
where $N_2$ is $\deg\left(f_2\right)$ and $n$ is $\deg(f)$.

For simplicity, put 
\[
\alpha_j':=\Phi_1^{-1}\left(\alpha_j\right)
\mbox{ and }
\beta_k':=\Phi_2^{-1}\left(\beta_k\right)
\]
where $j=1,\ldots,N_1$ and $k=1,\ldots,N_2$.
Introduce $q$ for the interpolation as follows
\begin{equation}
q(v):=\prod_{j=1}^{N_2} \frac{1-\overline{\beta_j'} \; v}{v-\beta_j'}
\; \cdot \; \prod_{j=1}^{N_5-N_2-2} 
\frac{1-\overline{\gamma}_j v}{v-\gamma_j}
\;\cdot(v-1)^2
\label{interpolqdef}
\end{equation}
where
$\gamma_j$, $j=1,\ldots,N_5-N_2-2$
are from $\Phi_2^{-1}(Z_2)$,
the first $N_4$ coincide with the $\zeta_j$'s (i.e. $\zeta_j=\gamma_j'=\Phi_2(\gamma_j)$ for $j=1,\ldots,N_4$)
and are arbitrary anyway.

Modulus of $q$ can be written as follows
(when $|\xi|\ge 1$)
\begin{equation}
\left|q(\xi)\right|
= 
\exp\left(2\log|\xi-1|
+
\sum_{j=1}^{N_2} \log\left| B\left( \beta_j',\xi\right)\right|
+
\sum_{k=1}^{N_5-N_2-2} \log\left| B\left(\gamma_k, \xi\right)\right|\right)
\label{eq:mod_of_q}
\end{equation}
where $B(a,v)=\frac{1-\overline{a}v}{v-a}$.

Note that if $|\xi|=1$, then $|q(\xi)|\le 4$.

On $D_1\setminus \bD$, $|q(.)|$ can be estimated from below as follows.
We know that $D_1$ and $\Phi_2^{-1}(Z_2)$ are disjoint and they are of positive distance from one another.
Therefore there exists $\cfive>0$ such that for all $\beta\in \Phi_2^{-1}(Z_2)$ and $v\in D_1\setminus \bD$, 
the modulus of the derivative of the Blaschke factor at $v$ with pole at 
$\beta$ is at least $\cfive$: $\left|B'(\beta,v)\right|\ge \cfive$.
This implies that for all $\beta\in \Phi_2^{-1}(Z_2)$ and $v\in D_1\setminus \bD$
\begin{equation}
\exp g_{\bD^*}\left(v,\beta\right)
\ge 1+\cfive \left(|v|-1\right).
\label{est:expGreen_lower}
\end{equation}
Moreover there is a similar upper estimate. That is, there exists $\cnineteen>0$ such
that for all $\beta\in \Phi_2^{-1}(Z_2)$ and $v\in D_1\setminus \bD$
\begin{equation}
\exp g_{\bD^*}\left(v,\beta\right)
\leq 1+\cnineteen \left(|v|-1\right).
\label{est:expGreen_upper}
\end{equation}

Multiplying \eqref{est:expGreen_lower} for all Blaschke factors (appearing in \eqref{interpolqdef}), we obtain
for $|v|=1+\delta_1$,
\begin{equation}
\left|q(v)\right|  \ge
\left(1+\cfive \delta_1 \right)^{N_5 -2} \delta_1^2.
\label{est:qlower_rough}
\end{equation}
We also have a sharper lower estimate:
\begin{equation}
\left|q(v)\right|  
\ge
\left(1+\cfive \delta_1 \right)^{N_5-N_2 -2}  
\;\delta_1^2\;
\prod_{j=1}^{N_2} 
\left|B\left(\beta_j',v\right)\right|.
\label{est:qlower_fine}
\end{equation}

We define the approximating rational function (see e.g. 
\cite{Walsh} chapter VIII)
\[
r_{1,N_5}\left(v\right):=\frac{1}{2\pi i}
\int_{\Gamma_{1}}\frac{\varphi_{1e}\left(\xi\right)}{q\left(\xi\right)}\frac{q\left(v\right)-q\left(\xi\right)}{v-\xi}d\xi.
\]
It is well known that $r_{1,N_5}$ does not depend on $\Gamma_{1}$.
Since $1$ is a double zero of $q$, therefore $r_{1,N_5}$ and $r_{1,N_5}'$
coincide there with $\varphi_{1e}$ and $\varphi_{1e}'$ respectively.

The error of the approximating rational function $r_{1,N_5}$  can be written as 
\begin{equation}
\varphi_{1e}\left(v\right)-r_{1,N_5}\left(v\right) 
=  \frac{1}{2\pi i}\int_{\Gamma_{1}}\frac{\varphi_{1e}\left(\xi\right)}{\xi-v}\frac{q\left(v\right)}{q\left(\xi\right)}d\xi,
\label{eq:p_N_one_error_est}
\end{equation}
where $\Gamma_{1}=\left\{ v\in\bC:\,|v|=1+\delta_{1}\right\} $ and
recall that $\varphi_{1e}$ is holomorphic in $\left\{ |v|<1+\delta_{1}\right\} $
and is continuous on $\left\{ |v|\leq1+\delta_{1}\right\} $.

For
$v\in\bD$ the error can be estimated as follows
\begin{multline*}
\left|\varphi_{1e}\left(v\right)-r_{1,N_5}\left(v\right)\right|
=
\frac{1}{2\pi}\int_{\Gamma_{1}}\frac{1}{\left|\xi-v\right|}\left|d\xi\right|\cdot\left\Vert \varphi_{1e}\right\Vert _{D_{1}}4
\frac{1}{\left(1+\cfive\delta_{1}\right)^{N_5-2}}
\frac{1}{\delta_{1}^{2}}
\\ \leq
4\frac{1+\delta_{1}}{\delta_{1}^{3}}
\left\Vert \varphi_{1e}\right\Vert _{D_{1}}
\frac{1}{\left(1+\cfive\delta_{1}\right)^{N_5-2}}
=4\frac{1+\delta_{1}}{\delta_{1}^{3}}
\frac{O\left(\log(n)e^{\cfour N_4}\right)}
{\left(1+\cfive\delta_{1}\right)^{N_5 -2}}
\left\Vert f\right\Vert _{\Gamma}
\end{multline*}
which tends to $0$ as $n\rightarrow\infty$, because $N_5/N_4\rightarrow\infty$ and $\delta_1$ is fixed, that is
\[
\frac{e^{\cfour N_4}}
{\left(1+\cfive\delta_{1}\right)^{N_5-2 }}
=
\exp\left(\cfour N_4-\log\left(1+\cfive\delta_{1}\right) N_5\left(1+o\left(1\right)\right)\right)\rightarrow 0.
\]

So, $r_{1,N_5}$ is a rational function with poles in $\Phi_{2}^{-1}\left(Z_{2}\right)$
only and we know that 
\begin{equation}
\left\Vert \varphi_{1e}-r_{1,N_5}\right\Vert _{\partial\bD}
=o\left(1\right)\left\Vert f\right\Vert _{\Gamma}
\label{eq:phi_one_e_sup_err_est}
\end{equation}
where $o\left(1\right)$ is independent of $f$ and depends only on
$\Gamma$ and tends to $0$ as $n\rightarrow\infty$, furthermore
\begin{equation}
\varphi_{1e}'\left(1\right)
=r_{1,N_5}'\left(1\right).
\label{eq:phi_one_e_deriv_at_one}
\end{equation}

\subsubsection{Approximating the term with poles outside}
Now we interpolate and approximate $\left( Q \cdot f_{2}\right)\circ\Phi_{1}$. 

We have the following description of the growth of Green's function. 
\begin{lem}
\label{lem:green_est}
There exists $\csix>0$ depending on $G_{2}$ only and 
is independent of $n$ and $f$ such that for all $1\le\left|v\right|\le1+\delta_{1}$
and $ \beta\in Z_2$ 
we have
\begin{equation}
g_{G_{2}}\left(\Phi_{1}\left(v\right),\beta\right)
\le
\csix\left(\left|v\right|-1\right).\label{eq:green_est_away}
\end{equation}
\end{lem}
\begin{proof}

We can express Green's function in the following way
for $u\in G_{2}$,
\[
g_{G_{2}}\left(u,\beta\right)
=
\log\left|B\left(\beta, \Phi_2^{-1}(u)
\right)\right|
\]
where $B\left(\beta, v \right)=\frac{1-\overline{\beta}\,v}{v-\beta}$ is Blaschke factor.
Hence
\[
g_{G_{2}}\left(\Phi_1(v),
\beta\right)
=
\log\left|B\left(\beta, \Phi_2^{-1}\circ \Phi_1(v)
\right)\right|.
\]
Here we use that 
there exists $\csixteen$ such that
for all $1\le |v|  \le 1+\delta_1$,
\[
\left|\left( \Phi_2^{-1}\circ \Phi_1(.)\right)' (v)\right|
\le \csixteen
\]
and there exists $\cseventeen$ such that
for all $1\le |v| \le 1+\delta_1$,
and $\beta\in Z_2$,
\[
\left|B\left( \beta, v\right)' \right|
\le \cseventeen.
\]
Finally, the directional derivative
of $g_{\bD^*}$ at $v$ from direction $v_1=v/|v|$ can be estimated as
\begin{equation*}
\frac{\partial}{\partial v_1 } g_{\bD^*}(v,\beta) 
\le
\left|B\left( \beta, .\right)' \vert_{\Phi_2^{-1}\circ \Phi_1(v)} \right|
\left|\left( \Phi_2^{-1}\circ \Phi_1(.)\right)' (v)\right|
\le \csixteen \cseventeen
\end{equation*}
and integrating it along the radial
ray $\left[v_1,v\right]$,
we obtain \eqref{eq:green_est_away}.

\end{proof}

Now we give the approximating rational function as follows
\[
r_{2,N_5}\left(v\right)
:= \frac{1}{2\pi i}\int_{\Gamma_2}\frac{\left(Q\cdot f_{2}\right)\circ\Phi_{1}\left(\xi\right)}{q\left(\xi\right)}\frac{q\left(v\right)-q\left(\xi\right)}{v-\xi}d\xi
\]
where $\Gamma_2$ can be arbitrary as long as $\bD\subset\inter\Gamma_2$ and
$\Gamma_2\subset D_{1}$, and $q$ is defined above. 
We remark that
we use the same interpolating points, but we need a different $\Gamma_2$
for the error estimate.

Now we construct $\Gamma_{2}$ for the estimate and investigate
the error. We use 
$\delta_{2,1}=1/(2n)$, 
 $\delta_{2,3}=\min\left(n^{-2/3}, 
 \delta_1 \right)$. 
We give four Jordan arcs 
that will make up $\Gamma_{2}$. Let $\Gamma_{2,3}$ be the (shorter,
circular) arc between $v_{1}^{*}\left(\delta_{2,3}\right)$ and $v_{2}^{*}\left(\delta_{2,3}\right)$,
$\Gamma_{2,1}$ be the longer circular arc between $v_{1}^{*}\left(\delta_{2,3}\right)\frac{1+\delta_{2,1}}{1+\delta_{2,3}}$
and $v_{2}^{*}\left(\delta_{2,3}\right)\frac{1+\delta_{2,1}}{1+\delta_{2,3}}$,
\begin{eqnarray*}
\Gamma_{2,2}: & = & \left\{ v:\ 1+\delta_{2,1}\le\left|v\right|\le1+\delta_{2,3},\ \arg v=\arg\left(v_{1}^{*}\left(\delta_{2,3}\right)\right)\right\} 
\end{eqnarray*}
 and similarly 
\begin{eqnarray*}
\Gamma_{2,4}: & = & \left\{ v:\ 1+\delta_{2,1}\le\left|v\right|\le1+\delta_{2,3},\ \arg v=\arg\left(v_{2}^{*}\left(\delta_{2,3}\right)\right)\right\} 
\end{eqnarray*}
 be the two segments connecting $\Gamma_{2,1}$ and $\Gamma_{2,3}$.
Finally let $\Gamma_{2}$ be the union of $\Gamma_{2,1}$, $\Gamma_{2,2}$,
$\Gamma_{2,3}$ and $\Gamma_{2,4}$. The figure 
\ref{fig:gamma2}
depicts these arcs and $K_{v}^{*}$ defined above. 
\begin{figure}
\begin{center}
\includegraphics[keepaspectratio,width=0.6\textwidth]{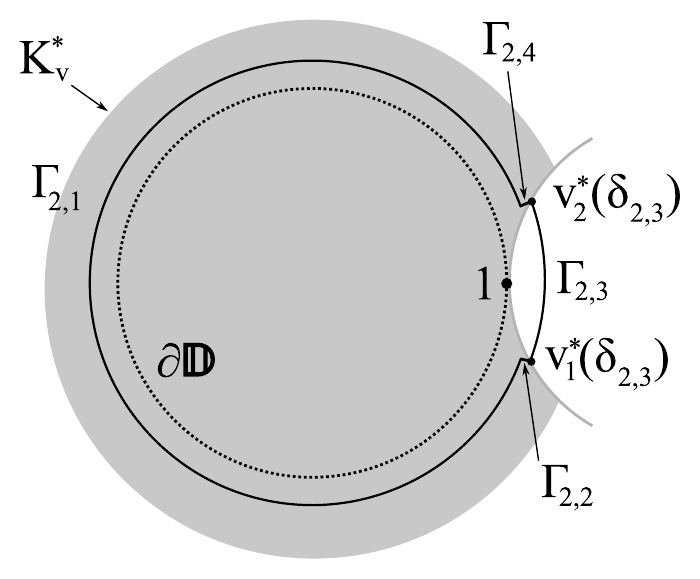}
\end{center}
\caption{The integration curve $\Gamma_2$}
\label{fig:gamma2}
\end{figure}

We estimate the error of $r_{2,N_5}$
to $\left(Q\cdot f_{2}\right)\circ\Phi_{1}$
on each integral separately:
\begin{multline*}
\left(Q\cdot f_{2}\right)\circ\Phi_{1}\left(v\right)
-r_{2,N_5}\left(v\right)
=\frac{1}{2\pi i}\int_{\Gamma_{2}}\frac{\left(Q\cdot f_{2}\right)\circ\Phi_{1}\left(\xi\right)}{\xi-v}\frac{q\left(v\right)}{q\left(\xi\right)}d\omega\\
=\frac{1}{2\pi i}\left(\int_{\Gamma_{2,1}}+\int_{\Gamma_{2,2}}+\int_{\Gamma_{2,3}}+\int_{\Gamma_{2,4}}\right).
\end{multline*}
For the first term, we use  Bernstein-Walsh estimate 
for the rational function $f_{2}$ on $G_{2}$ and the fast decreasing rational function
$Q$ as follows. 
If $v\in\Gamma_{2,1}$, then with \eqref{eq:green_est_away},
$g_{G_{2}}\left(\Phi_{1}\left(v\right),\beta\right)
\le
\csix\delta_{2,1}
=\csix/(2n)$
(uniformly in $\beta\in Z_2$),
therefore
\begin{multline*}
\left|f_{2}\left(\Phi_{1}\left(v\right)\right)\right|
\le
\left\Vert f_{2}\right\Vert _{\Gamma}\exp\left(n\frac{\csix}{2n}\right)
\le
\left\Vert f\right\Vert _{\Gamma}\constgg\left(G_{1}\right)\log\left(n\right)e^{\csix/2}
\\= 
O\left(\log\left(n\right)\right)\left\Vert f\right\Vert _{\Gamma}
\end{multline*}
where we used (\ref{eq:f_two_norm_est}). Now we use the fast decreasing
property of $Q$ as follows. 

We use  
\[
\cnine := \inf\left\{ \left|\Phi_{1}'\left(v\right)\right|:\ v\in D_{1}\right\} >0.
\]
Hence, if $v\in\Gamma_{2,1}\cup \Gamma_{2,2} \cup \Gamma_{2,4}$, then by \eqref{eq:w_star_dist},
$|v-1|\ge 1/2\, \sqrt{\delta_{2,3}}$,
and then
\begin{equation}
|u-u_0|\ge 
\frac{\cnine}{2} \sqrt{\delta_{2,3}}
\label{est:u_dist_for_fastd}
\end{equation}
where $u=\Phi_1(v)$ and $u_0=\Phi_1(1)$
and therefore
the fast decreasing rational function
$Q$ is small, see \eqref{eq:Q_fast_decreases}, in other words,
\begin{equation}
\left| Q\left(\Phi_1(v) \right)\right|
\leq
\cthirteen \exp\left(
-\cfourteen N_3 \left|\cnine/2\, \sqrt{\delta_{2,3}} \right|^\tau\right).
\label{est:Q_on_gamma124}
\end{equation}

Therefore  we can write
\begin{multline*}
\left|\left(Q\cdot f_{2}\right)\left(\Phi_{1}\left(v\right)\right)\right|
\le 
O\left(\log\left(n\right)\right)
\left\Vert f\right\Vert _{\partial G_{2}}
\cthirteen \exp\left(
-\cfourteen N_3 \left|\cnine/2\, \sqrt{\delta_{2,3}} \right|^\tau\right)
\\ \le 
O\left(\frac{\log\left(n\right)}{\exp\left(\cfourteen (\cnine/2)^\tau 
\; N_3 n^{-\tau/3}\right)}\right)
\left\Vert f\right\Vert _{\Gamma}
=: E_{2,1} 
\left\Vert f\right\Vert _{\Gamma}
\end{multline*}
where the coefficient tends to $0$ 
because, with the definition of $\delta_{2,3}$,
we can write $N_3 \delta_{2,3}^{\tau/2} = n^{3/4} \min\left(n^{-\tau/3}, 
\delta_1^{\tau/2}\right)$
and $n^{3/4} n^{-\tau/3}\rightarrow\infty$ since $\tau=4/3$. 

Integrating along $\Gamma_{2,1}$, we can write for $v\in\bD$
\begin{multline*}
\left|\frac{1}{2\pi i}
\int_{\Gamma_{2,1}}\frac{\left(Q\cdot f_{2}\right)\circ\Phi_{1}\left(\xi\right)}{\xi-v}\frac{q\left(v\right)}{q\left(\xi\right)}d\xi\right|
\\ \le
\frac{1}{2\pi}
\int_{\Gamma_{2,1}}\frac{1}{\left|\xi-v\right|}
E_{2,1}
\left\Vert f\right\Vert_{\Gamma}4\frac{1}{\left(1+\cfive\delta_{2,1}\right)^{N_{2}}\delta_{2,1}^{2}}\left|d\omega\right|
\\ \le
\frac{2}{\pi}\frac{2\pi\left(1+\delta_{2,1}\right)}{\left(1+\cfive\delta_{2,1}\right)^{N_2}\delta_{2,1}^{3}}
E_{2,1}
\left\Vert f\right\Vert_{\Gamma}
=
O\left(n^3\right) E_{2,1}
\left\Vert f\right\Vert _{\Gamma}
\end{multline*}
here we used $\delta_{2,1}=1/(2n)$
and $n^3 E_{2,1}\rightarrow 0$.

We estimate the third term, the integral on $\Gamma_{2,3}$, as follows
for $v\in\bD$
\begin{multline}
\left|\frac{1}{2\pi i}
\int_{\Gamma_{2,3}}\frac{\left(Q\cdot f_{2}\right)\circ\Phi_{1}\left(\xi\right)}{\xi-v}\frac{q\left(v\right)}{q\left(\xi\right)}d\omega\right|
\\ \le
\frac{1}{2\pi}
\int_{\Gamma_{2,3}}4\frac{1}{\left|\xi-v\right|}\left|\left(Q\cdot f_{2}\right)\left(\Phi_{1}\left(\xi\right)\right)\right|\frac{1}{\left|q\left(\xi\right)\right|}\left|d\xi\right|.\label{eq:third_term_main_est}
\end{multline}
Here, $\left|\xi\right|=1+\delta_{2,3}$
and $\left|v-\xi\right|\ge\delta_{2,3}$.
Roughly speaking, $f_{2}$ grows and this time $Q$ grows too  and only $\left|q\left(\xi\right)\right|^{-1}$ decreases. 
We are going to estimate the growths
with the help of Bernstein-Walsh estimate and Blaschke factors.

For simplicity, put $F(\xi):=\Phi_2^{-1}\circ \Phi_1 \left(\xi\right)$, hence
$F(1)=1$, $F'(1)=1$.
This latter holds because $\Phi_1(1)=\Phi_2(1)=u_0$.
Moreover, $\Phi_1(\partial \bD) =\Phi_2(\partial \bD)=\Gamma$,
hence $\mathrm{arg} \Phi_1'(1)=\mathrm{arg} \Phi_2'(1)$ and together with
  $\left|\Phi_1'(1)\right|=\left|\Phi_2'(1)\right|=1$,
it implies that $\Phi_1'(1)=\Phi_2'(1)$.
Therefore $\left(\Phi_2^{-1}\circ \Phi_1\right)'(1)=1$.
Moreover, $\exp g_{G_2}\left(\Phi_1(\xi), \Phi_2(\gamma)\right)=B(\gamma,F(\xi))$ (if $|\gamma|> 1$).

First,
Bernstein-Walsh
estimate  for $f_{2}$ on $G_{2}$ yields for $\xi\in D_1\setminus \bD$
(and in particular, if $\xi\in\Gamma_{2,3}$) that
\begin{multline}
\left|f_{2}\left(\Phi_{1}\left(\xi\right)\right)\right|
\le
\left\Vert f_{2}\right\Vert _{\Gamma}
\exp\left(
\sum_{j=1}^{N_{2}}g_{G_{2}}\left(\Phi_{1}\left(\xi\right),\beta_j \right)\right)
\\ = 
\left\Vert f_{2}\right\Vert_{\Gamma}
\prod_{j=1}^{N_2}
\left|B\left( \beta_j', F \left(\xi\right) \right)  \right|. 
\label{est:f2_on_gamma23}
\end{multline}

As for $q$, we use \eqref{eq:w_star_dist}
and \eqref{eq:mod_of_q}, hence
 for $\xi \in D_1\setminus \bD$
\begin{multline}
\frac{1}{\left|q\left(\xi\right)\right|}
\le
\frac{1}{|\xi-1|^{2}}
\exp\left(
- \sum_{k=1}^{N_5-N_2-2} \log\left| B\left(\gamma_k, \xi\right)\right|
-\sum_{j=1}^{N_{2}}
\log\left|B\left(\beta_j',\xi \right)\right|\right)
\\=
\frac{1}{|\xi-1|^{2}}
\prod_{k=1}^{N_5-N_2 -2}
\left|B\left( \gamma_k, \xi \right)  \right| ^{-1}
\prod_{j=1}^{N_2}
\left|B\left( \beta_j', \xi \right)  \right| ^{-1}.
\label{est:q_on_gamma23_pointw}
\end{multline}

As for $Q$, we use  Bernstein-Walsh
estimate for $Q$ on $G_{1}\cup\partial G_{1}$
and that $G_{1}\cup\partial G_{1}\subset K_{u}^{*}$. Therefore, $\left\Vert Q\right\Vert _{\Gamma}=1$
and 
\begin{equation}
\left|Q\left(\Phi_{1}\left(\xi\right)\right)\right|
\le
\left\Vert Q\right\Vert_{\Gamma}
\exp\left(\sum_{j=1}^{N_4} g_{G_{2}}\left(\Phi_{1}\left(\xi\right),\zeta_j\right)\right)
 \le
\prod_{j=1}^{N_4}
\left|B\left(\zeta_j, F(\xi)\right)\right|
.
\label{est:fastq_on_gamma23}
\end{equation}

Now we are going to multiply together these estimates.
Consider the quotients
\begin{equation*}
H(\beta',\xi)=H(\xi) :=
\frac{B\left( \beta', F \left(\xi\right) \right)  }{B\left( \beta', \xi \right)}
\end{equation*}
for $\beta'\in \bD^*$.

First, $\log(H(\xi))'$ is continuous and holomorphic (near $1$, using the main branch of the logarithm since $H(1)=1$), $\log(H)'(1)=0$ because $F'(1)=0$. 
So there exists $\clogh>0$ such that $|\log(H(\xi))'|\leq \clogh |\xi-1|$ and this is uniformly true for all $\beta'\in \Phi_2^{-1}(Z_2)$ (since $\Phi_2^{-1}(Z_2)$ is compact in $\bD^*$). 
Now, taking the real part and 
integrating along radial rays (see also the proof of Lemma 9 in \cite{JMAA}), 
we get
\begin{equation}
\log\left|H(\xi)\right|\leq 
\clogh \left(|\xi|-1\right)
\left|\xi-1\right|.
\label{eq:loghest}
\end{equation}

Applying this estimate for the product of \eqref{est:f2_on_gamma23},
\eqref{est:q_on_gamma23_pointw},
and \eqref{est:fastq_on_gamma23}
we can write

\begin{multline*}
\left|f_{2}\left(\Phi_{1}\left(\xi\right)\right)\right|
\frac{1}{\left|q\left(\xi\right)\right|}
\left|Q\left(\Phi_{1}\left(\xi\right)\right)\right|
\le 
\left\Vert f_{2}\right\Vert_{\Gamma}
\prod_{j=1}^{N_2} \left| B\left( \beta_j',F(\xi)\right)\right|
\\ \cdot\  
\frac{1}{|\xi-1|^{2}}
\prod_{k=1}^{N_5-N_2 -2}
\left|B\left( \gamma_k, \xi \right)  \right| ^{-1}
\prod_{j=1}^{N_2}
\left|B\left( \beta_j', \xi \right)  \right|^{-1}
 \ \cdot \  
\prod_{j=1}^{N_4}
\left|B\left(\zeta_j, F(\xi)\right)\right|
\\ \le 
\left\Vert f_{2}\right\Vert_{\Gamma}
\frac{1}{\delta_{2,3}^{2}}
\frac{1}{\left(1+\cfive\delta_{2,3}\right)^{
N_5-N_2-2 -N_4 
}}
\\ \cdot  \exp\left(\sum_{j=1}^{N_2}\log \left|  H\left( \beta_j', \xi\right)\right|
+\sum_{j=1}^{N_4}\log\left|  H\left( \zeta_j, \xi\right) \right|\right)
\\ \le 
\left\Vert f_{2}\right\Vert_{\Gamma}  
n^{4/3} \cdot 
 \left(1+\cfive n^{-2/3}\right)^{- n^{4/5}}
\cdot \exp\left(\clogh (N_2+N_4)\delta_{2,3}\, 2\sqrt{\delta_{2,3}}\right)
\\ \le 
\left\Vert f_{2}\right\Vert_{\Gamma}
n^{4/3}e^{-\cfive/2\, n^{2/15}}
\exp\left(O\left( n\cdot \frac{1}{n^{2/3}}\frac{1}{n^{1/3}}\right) \right)
\\ =
\left\Vert f_{2}\right\Vert_{\Gamma} 
 n^{4/3} e^{-\cfive/2\, n^{2/15}} O(1)
 =: \left\Vert f_{2}\right\Vert_{\Gamma} 
E_{2,3}
\end{multline*}
where we used $N_5-N_2-2=N_4+[ n^{4/5}] -2$ and \eqref{eq:loghest} at the second step, the definition of $\delta_{2,3}$ and \eqref{eq:w_star_dist} at the third step, 
again the the definition of $\delta_{2,3}$ and that $n$ is large (hence $\left(1+\frac{\cfive}{n^{2/3}}\right)^{n^{2/3}} \ge  e^{\cfive/2}$) at the fourth step.

Therefore, the integral over $\Gamma_{2,3}$
can be written as (see \eqref{eq:third_term_main_est})
\begin{equation}
\left|\frac{1}{2\pi i}
\int_{\Gamma_{2,3}}\frac{\left(Q\cdot f_{2}\right)\circ\Phi_{1}\left(\xi\right)}{\xi-v}\frac{q\left(v\right)}{q\left(\xi\right)}d\omega\right| 
\leq 
\left\Vert f_{2}\right\Vert _{\Gamma}
E_{2,3} \frac{8}{ \delta_{2,3}}
\end{equation}
where we used that 
$|q(v)|\le 4$ (if $|v|=1$), and 
the length of $\Gamma_{2,3}$ is at most $4\pi$.

\medskip

For $\Gamma_{2,2}$ and $\Gamma_{2,4}$, we apply the same estimate
which we detail for $\Gamma_{2,2}$ only. We again start with the
integral for $v\in\bD$
\begin{multline}
\left|\frac{1}{2\pi i}
\int_{\Gamma_{2,2}}
\frac{\left(Q\cdot f_{2}\right)\circ\Phi_{1}\left(\xi\right)}
{v-\xi}
\frac{q\left(v\right)}
{q\left(\xi\right)}d\xi\right|
\\ \le
\frac{1}{2\pi}
\int_{\Gamma_{2,2}}
4\frac{1}{\left|v-\xi\right|}
\left|\left(Q\cdot f_{2}\right)\left(\Phi_{1}\left(\xi\right)\right)\right|
\frac{1}{\left|q\left(\xi\right)\right|}\left|d\xi\right|.\label{eq:second_term_main_est}
\end{multline}
Since $\xi\in\Gamma_{2,2}$, we can rewrite it in the form $\xi=\left(1+\delta\right)v_{1}^{*}$
where $\delta_{2,1}\le\delta\le\delta_{2,3}$ (with 
$v_{1}^{*}=v_{1}^{*}\left(\delta_{2,3}\right)$). 
We use 
similar steps 
to estimate $f_{2}$  and $q$ and $Q$.

We use the estimate  \eqref{est:f2_on_gamma23} for $|f_2|$,
\eqref{est:fastq_on_gamma23} for $1/|q|$ 
and \eqref{est:Q_on_gamma124} for $Q$.
We also use $\left|\xi\right|=1+\delta$,
so 
\begin{multline*}
\left|f_{2}\left(\Phi_{1}\left(\xi\right)\right)\right|
\frac{1}{\left|q\left(\xi\right)\right|}
\left|Q\left(\Phi_{1}\left(\xi\right)\right)\right|
\le 
\left\Vert f_{2}\right\Vert _{\Gamma}
\exp\left(\sum_{j=1}^{N_2} \log \left|H(\beta_j',\xi)\right|\right)
\\ \cdot
\frac{1}{|\xi-1|^{2}}
\frac{1}{\left(1+\cfive\delta\right)^{N_5-N_2 -2}}
\ \cdot\, 
\cthirteen \exp\left(
-\cfourteen (\cnine/2)^\tau \,
N_3 \delta_{2,3}^{\tau/2}\right)
\le
\end{multline*}
which we continue using $1/2\, \sqrt{\delta}\le |\xi-1|\le 2\sqrt{\delta}$ (see \eqref{eq:w_star_dist})
and $1/(2n)=\delta_{2,1}\le \delta\le \delta_{2,3}$
and \eqref{eq:loghest} for the sum. 
So
\begin{equation*}
\le
\left\Vert f_{2}\right\Vert _{\Gamma}
8 
n^2 \;1\;
\exp\left(\clogh N_2\delta_{2,3}\, 2\sqrt{\delta_{2,3}}\right)
\cthirteen \exp\left(
-\cfourteen (\cnine/2)^\tau \,
N_3 \delta_{2,3}^{\tau/2}\right)
 \le 
\end{equation*}
which tends to $0$ if $N_3\delta_{2,3}^{\tau/2}\rightarrow\infty$ and we can continue this estimate using $\delta_{2,3}= n^{-2/3}$, so
\begin{multline*}
\leq
\left\Vert f_{2}\right\Vert _{\Gamma}
8 
\cthirteen e^{2\clogh} \; 
n^2 \exp\left(
-\cfourteen \cnine^\tau 2^{-3\tau/2} \,
N_3 n^{-\tau/2}\right)
\\
\leq 
\left\Vert f_{2}\right\Vert _{\Gamma}
8 
\cthirteen e^{2\clogh} \; 
n^2 \exp\left(
-\cfourteen \cnine^\tau 2^{-3\tau/2} \,
 n^{1/12}\right)
=:
\left\Vert f_{2}\right\Vert _{\Gamma}
E_{2,2}.
\end{multline*}
Now continuing \eqref{eq:second_term_main_est}
we write
\begin{equation*}
\le \frac{1}{2\pi}
\int_{\Gamma_{2,2}}
4 \cdot 2n \left\Vert f_{2}\right\Vert _{\Gamma}
E_{2,2}
|d\xi|
\leq
\frac{2}{\pi}
\; n \delta_{2,3} E_{2,2}
\left\Vert f_{2}\right\Vert _{\Gamma}
\end{equation*}
where $\delta_{2,3}= n^{-2/3}$ 
and $n E_{2,2}\rightarrow 0$.

Summarizing these estimates on $\Gamma_{2,1}$, $\Gamma_{2,3}$ and
$\Gamma_{2,2}$ (and also on $\Gamma_{2,4}$), 
   and using \eqref{eq:f_two_norm_est} with the exponential decay of $E_{2,1},\ldots,E_{2,4}$, 
we have uniformly for
$\left|v\right|\le1$, 
\[
\left|r_{2,N_5}\left(v\right)-\left(Q\cdot f_{2}\right)\circ\Phi_{1}\left(v\right)\right|
=
o\left(1\right)\left\Vert f\right\Vert _{\Gamma}
\]
where $o\left(1\right)$ tends to $0$ as $n\rightarrow\infty$ but
it is independent of $f$ and $f_{2}$. 
Obviously, $r_{2,N_5}$ is a
rational function with poles in  $\Phi_{2}^{-1}\left(Z_{2}\right)$
only with $\deg\left( r_{2,N_5}\right)=N_5=\left(1+o\left(1\right)\right)n$
and therefore we uniformly have for $\left|v\right|\le1$
\[
\left|r_{2,N_5}\left(v\right)-\left(Q\cdot f_{2}\right)\circ\Phi_{1}\left(v\right)\right|
=
o\left(1\right)\left\Vert f\right\Vert _{\Gamma},
\]
that is,
\begin{equation}
\left\Vert r_{2,N_5}-\left(Q\cdot f_{2}\right)\circ\Phi_{1}\right\Vert _{\partial\bD}
=
o\left(1\right)\left\Vert f\right\Vert _{\Gamma}
\label{eq:sup_norm_err_est_two}
\end{equation}
where $o\left(1\right)$ tends to $0$ as $n\rightarrow\infty$ but
it is independent of $f$.
Since $1$ is double zero of $q$, so we have 
\begin{equation}
r_{2,N_5}'\left(1\right)
=
\left(\left(Q\cdot f_{2}\right)\circ\Phi_{1}\right)'\left(1\right).\label{eq:p_two_n_psi_deriv_at_one}
\end{equation}

\subsubsection{A similar rational function and final estimates}

Consider the ``constructed'' rational function 
\[
h\left(v\right)
:=
\varphi_{1r}\left(v\right)+r_{1,N_5}\left(v\right)+r_{2,N_5}\left(v\right).
\]
Recall that the poles of $\varphi_{1r}$
are $\alpha_j'$, $j=1,\ldots,N_1$,
the poles of both
$r_{1,N_5}$ and $r_{2,N_5}$ 
$\beta_j'$, $j=1,\ldots,N_2$
and the poles of $Q$ are 
$\zeta_j$, $j=1,\ldots,N_4$ (counting
multiplicities).
Hence, this function $h$ has poles 
at $\alpha_j'=\Phi_1^{-1}\left(\alpha_j\right)$ (and with exactly the same multiplicities),
and $h$ has poles at $\beta_j'$-s
and here the multiplicities may change a bit. 
We use the
identity 
\[
f\circ\Phi_{1}=\left(Q\cdot f+\left(1-Q\right)\cdot f\right)\circ\Phi_{1}
\]
to calculate the derivatives as follows 
\[
\left(\left(\left(1-Q\right)\cdot f\right)\circ\Phi_{1}\right)'\left(1\right)
=
\left(\left(1-Q\right)'\cdot f\right)\left(u_{0}\right)\cdot\Phi_{1}'\left(1\right)+\left(\left(1-Q\right)\cdot f'\right)\left(u_{0}\right)\cdot\Phi_{1}'\left(1\right)
\]
where the second term is zero because of the fast decreasing rational function
($Q\left(u_{0}\right)=1$) and for the first term we can apply the rough Bernstein type inequality (Proposition \ref{prop:rough_berns}) in the following way ($\left\Vert 1-Q\right\Vert_{\Gamma}\le2$):
\[
\left|\left(1-Q\right)'\left(u_{0}\right)\right|
\le
\deg\left(Q\right)\, \crbone \, 2
=o(n)
\]
where $\deg(Q)=N_4\le n^{3/4}$. 
Here we use that 
\begin{equation}
\ctwentyone := \inf\left\{
\frac{\partial}{\partial n_2 (u)} g_{G_2}(u,\beta),
\frac{\partial}{\partial n_1 (u)} g_{G_1}(u,\alpha) :
\  u\in \Gamma,
\alpha \in Z_1, \beta\in Z_2
\right\} >0
\label{est:ndg_lower}
\end{equation}
because  $g_{G_1}(u,\alpha)=\log\left|B\left(\Phi_1^{-1}(u),\Phi_1^{-1}(\alpha)\right)\right|$, the derivatives of Blaschke factors are bounded away from $0$ in modulus ($\Phi_1^{-1}(Z_1)$ is closed and disjoint from the unit circle) 
and $\left|\Phi_1'(.)\right|$ is bounded away from $0$. Similarly for $g_{G_2}(.,.)$.  
There is a uniform upper estimate
for the normal derivatives:
\begin{equation}
\ctwentytwo := \sup\left\{
\frac{\partial}{\partial n_2 (u)} g_{G_2}(u,\beta),
\frac{\partial}{\partial n_1 (u)} g_{G_1}(u,\alpha) :
\  u\in \Gamma,
\alpha \in Z_1, \beta\in Z_2
\right\} <\infty.
\label{est:ndg_upper}
\end{equation}

Therefore
\begin{multline}
\left|\left(\left(1-Q\right)'\cdot f\right)
\left(u_{0}\right)\cdot
\Phi_{1}'\left(1\right)\right|
\le
\left\Vert f\right\Vert_{\Gamma}\crbone n^{3/4} 2
\\ \leq
\left\Vert f\right\Vert _{\Gamma}
\frac{2 n^{3/4} \crbone}{ \frac{1}{2} n \ctwentyone}
\max
\left( \sum_{j=1}^{N_2} \frac{\partial}{\partial n_{2}\left(u_{0}\right)}g_{G_{2}}\left(u_{0},\beta_j\right),
\sum_{j=1}^{N_1} \frac{\partial}{\partial n_{1}\left(u_{0}\right)}
g_{G_{1}}\left(u_{0},\alpha_j\right)\right)
\\ = 
o\left(1\right)\left\Vert f\right\Vert _{\Gamma}
\max
\left( \sum_{j=1}^{N_2} \frac{\partial}{\partial n_{2}\left(u_{0}\right)}g_{G_{2}}\left(u_{0},\beta_j\right),
\sum_{j=1}^{N_1} \frac{\partial}{\partial n_{1}\left(u_{0}\right)}
g_{G_{1}}\left(u_{0},\alpha_j\right)\right)
\label{est:one_minus_q_f_prime}
\end{multline}
where we used that
\begin{equation}
\frac{1}{2} n \ctwentyone
\leq
\max
\left( \sum_{j=1}^{N_2} \frac{\partial}{\partial n_{2}\left(u_{0}\right)}g_{G_{2}}\left(u_{0},\beta_j\right),
\sum_{j=1}^{N_1} \frac{\partial}{\partial n_{1}\left(u_{0}\right)}
g_{G_{1}}\left(u_{0},\alpha_j\right)\right).
\label{est:max_from_below}
\end{equation}
This way we need to consider
$\left(Q\cdot f\right)\circ\Phi_{1}$ only. 
The derivatives at $1$
of the original $f$ and $h$ coincide, 
because of  \eqref{eq:phi_decomp},
\eqref{eq:phi_one_e_deriv_at_one} and \eqref{eq:p_two_n_psi_deriv_at_one},
so 
\begin{equation}
h'\left(1\right)
=
\varphi_{1r}'\left(1\right)
+r_{1,N_5}'\left(1\right)
+r_{2,N_5}'\left(1\right)
=
\left(\left(Q\cdot f\right)\circ\Phi_{1}\right)'\left(1\right).\label{eq:h_qf_prime_at_one}
\end{equation}

As for the sup norms, we use \eqref{eq:phi_decomp},  \eqref{eq:phi_one_e_sup_err_est},
\eqref{eq:sup_norm_err_est_two}, 
so we write 
\begin{equation}
\left\Vert\left(Q\cdot f\right)\circ\Phi_{1}
-
h\right\Vert _{\partial\bD}
=
o\left(1\right)\left\Vert f\right\Vert _{\partial G_{2}}
\label{est:qf_h_sup_norm}
\end{equation}
where $o\left(1\right)$ tends to $0$ as $n=\deg(f)\rightarrow\infty$ but
it is independent of $f$, this follows from
the considerations after \eqref{eq:phi_one_e_sup_err_est} and
\eqref{eq:sup_norm_err_est_two}.

Now we apply  Borwein-Erdélyi inequality 
for $h$ as follows: 
\begin{multline}
\left|h'\left(1\right)\right|
\le
\left\Vert h\right\Vert_{\partial\bD}
\max\Big(
\sum_{j=1}^{N_1}
\frac{\partial}{\partial n_{1}\left(1\right)}
g_{\bD}\left(1,\alpha_j'\right)
,\\
\sum_{j=1}^{N_2}
\frac{\partial}{\partial n_{2}\left(1\right)}
g_{\bD^{*}}\left(1,\beta_j'\right)
+
\sum_{j=1}^{N_5-N_2-2}
\frac{\partial}{\partial n_{2}\left(1\right)}
g_{\bD^{*}}\left(1,\gamma_j\right)
\Big)
\label{est:main_est_for_h}
\end{multline}
where we also used the definition of $\gamma_j$'s (see \eqref{interpolqdef}).
Now we apply Proposition \ref{prop:green_deriv_unitdisk}
and \eqref{est:max_from_below}
with $N_5-N_2-2=o(n)$ and the uniform upper estimate \eqref{est:ndg_upper} 
so we can continue the main estimate (\ref{est:main_est_for_h})
\begin{multline*}
=
\left\Vert h\right\Vert _{\partial\bD} 
\max\Big( 
\sum_{j=1}^{N_1} \frac{\partial}{\partial n_{1}\left(u_{0}\right)}
g_{G_{1}}\left(u_{0},\alpha_j\right),
\\
\sum_{j=1}^{N_2} \frac{\partial}{\partial n_{2}\left(u_{0}\right)}g_{G_{2}}\left(u_{0},\beta_j\right)
+
\sum_{j=1}^{N_5-N_2-2}
\frac{\partial}{\partial n_{2}\left(u_0\right)}
g_{G_2}\left(u_0,\Phi_2^{-1}\left(\gamma_j\right)\right)
\Big)
\\ \le 
\left\Vert h\right\Vert _{\partial\bD}\left(1+o\left(1\right)\right)
\max\Big( 
\sum_{j=1}^{N_1} \frac{\partial}{\partial n_{1}\left(u_{0}\right)}
g_{G_{1}}\left(u_{0},\alpha_j\right),
\\
\sum_{j=1}^{N_2} \frac{\partial}{\partial n_{2}\left(u_{0}\right)}g_{G_{2}}\left(u_{0},\beta_j\right) \Big)
\leq
\left(1+o\left(1\right)\right)
\left\Vert f\right\Vert_{\Gamma}
\max\Big( 
\sum_{j=1}^{N_1} \frac{\partial}{\partial n_{1}\left(u_{0}\right)}
g_{G_{1}}\left(u_{0},\alpha_j\right),
\\
\sum_{j=1}^{N_2} \frac{\partial}{\partial n_{2}\left(u_{0}\right)}g_{G_{2}}\left(u_{0},\beta_j\right) \Big)
\end{multline*}
where in the last step we used \eqref{est:qf_h_sup_norm} and
the properties of $Q$.
Here, $o(1)$ tends to $0$ as $n=\deg(f)\rightarrow\infty$ but
it is independent of $f$, this follows from
the consideration \eqref{est:qf_h_sup_norm}
and $N_5-N_2-2=o(n)$.

Using that $\left|h'(1)\right|=\left| f'\left(u_0\right)\right|$ and summarizing these estimates, we have 
\begin{multline*}
\left|f'\left(1\right)\right|
\le
\left\Vert f\right\Vert_{\Gamma}
\left(1+o\left(1\right)\right)\\
\cdot\max\left(\sum_{j=1}^{N_1}
\frac{\partial}{\partial n_{1}\left(u_{0}\right)}
g_{G_{1}}\left(u_{0},\alpha_{j}\right),\ 
\sum_{j=1}^{N_2}
\frac{\partial}{\partial n_{2}\left(u_{0}\right)}
g_{G_{2}}\left(u_{0},\beta_j\right)\right)
\end{multline*}
which is the assertion of Theorem \ref{thm:analrational}.

\section{Proof of Theorem \ref{thm:analrationalarc}}

Here we use an analytic open-up
tool to transform the arc setting ($z$ plane) to the curve setting ($u$ plane).

For a Jordan curve $\Gamma$, $\mathrm{Int}\Gamma$ denotes the interior of $\Gamma$ and
$\mathrm{Ext}\Gamma$ denotes the exterior of $\Gamma$, 
$\mathrm{Ext}\Gamma := \bC_\infty \setminus \left(\Gamma \cup \mathrm{Int}\Gamma\right)$.

\begin{prop}
\label{prp:openup}
Let $\Gamma_0\subset\bC$ be an analytic Jordan arc. 
Then there exist a rational function 
$F$ and an analytic Jordan curve $\Gamma$ such that 
$F$ is a conformal bijection from $\mathrm{Int} \Gamma$ and from $\mathrm{Ext}\Gamma$ 
onto $\bC_\infty\setminus \Gamma_0$.
\end{prop}
This is a special case of \cite{JMAA}, Proposition 5 and actually can be established with the Joukowskii mapping
$z=J(u)=1/2(u+1/u)$ and using a suitable linear transformation.
This mapping is depicted on figure \ref{fig:openup}. 
\begin{figure}
\begin{center}
\includegraphics[keepaspectratio,width=\textwidth]{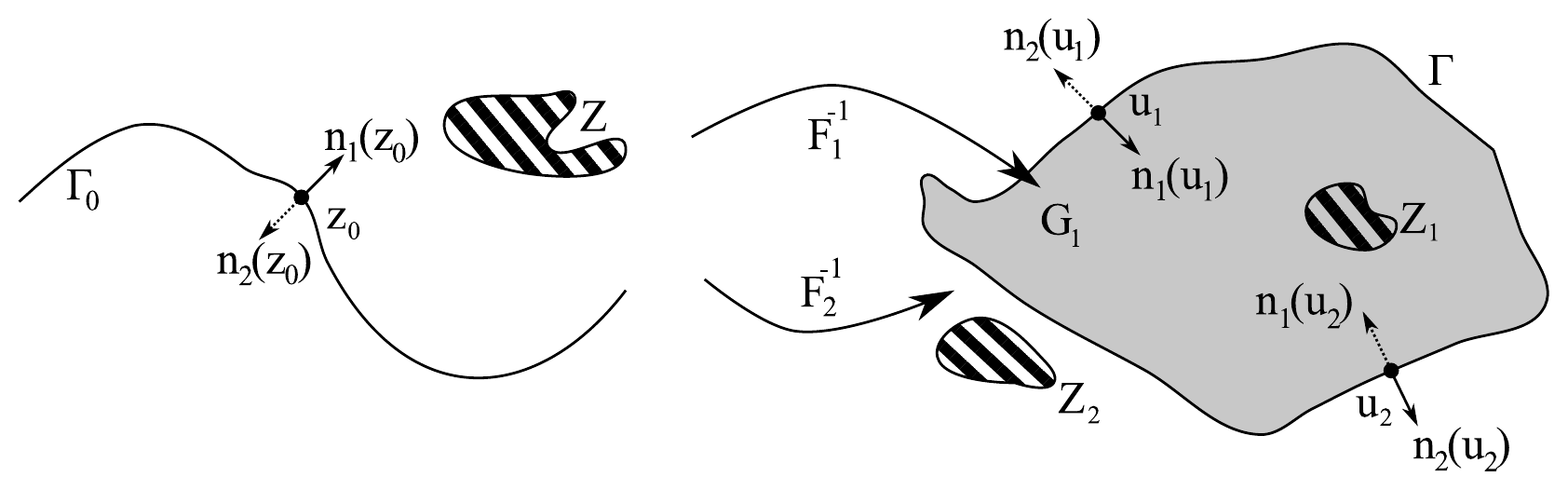}
\end{center}
\caption{The open-up}
\label{fig:openup}
\end{figure}

Denote the inverse of $z=F(u)$ restricted to $\mathrm{Int}\Gamma$ by $F_1^{-1}(z)=u$ and
that restricted to $\mathrm{Ext}\Gamma$
by $F_2^{-1}(z)=u$.

We need the mapping properties of $F$
as regards the normal vectors.
For the full details, we refer to \cite{JMAA} p. 879.
Briefly, there are exactly two $u_1,u_2\in \Gamma$, $u_1\ne u_2$ 
such that $F(u_1)=F(u_2)=z_0$.
Denote the normal vectors to $\Gamma$
pointing outward by $n_2(.)$ and the normal vectors pointing inward by $n_1(.)$.
By reindexing $u_{1}$ and $u_{2}$ and the exchanging $n_1(z_0)$ and $n_2(z_0)$, 
we may assume that the normal
vector $n_{2}\left(u_{1}\right)$ is mapped by $F$ to the normal
vector $n_{2}\left(z_0\right)$. This immediately implies that $n_{1}\left(u_{1}\right)$,
$n_{2}\left(u_{2}\right)$, $n_{1}\left(u_{2}\right)$ are mapped
by $F$ to $n_{1}\left(z_0\right)$, $n_{1}\left(z_0\right)$,
$n_{2}\left(z_0\right)$
respectively.

Moreover, we need to relate the normal derivatives of Green's functions as follows.
\begin{prop}
\label{prp:greenopenup}
Using the notations above, for the Green's
functions of $G_{2}:=\mathrm{Ext} \Gamma$ and $G_{1}:=\mathrm{Int}\Gamma$ and for $b\in\bC_{\infty}\setminus K$
we have 
\begin{multline*}
\frac{\partial}{\partial n_{1}\left(z\right)}
g_{\bCi\setminus \Gamma_0}\left(z,b\right)
=
\frac{\partial}{\partial n_{1}\left(u_{1}\right)}
g_{G_{1}}\left(u_{1},F_{1}^{-1}\left(b\right)\right)/\left|F'\left(u_{1}\right)\right|
\\ =
\frac{\partial}{\partial n_{2}\left(u_{2}\right)}
g_{G_{2}}\left(u_{2},F_{2}^{-1}\left(b\right)\right)/\left|F'\left(u_{2}\right)\right|
\end{multline*}
and, similarly for the other side,
\begin{multline*}
\frac{\partial}{\partial n_{2}\left(z\right)}
g_{\bCi\setminus \Gamma_0}\left(z,b\right)
=
\frac{\partial}{\partial n_{1}\left(u_{2}\right)}
g_{G_{1}}\left(u_{2},F_{1}^{-1}\left(b\right)\right)/\left|F'\left(u_{2}\right)\right|
\\ =
\frac{\partial}{\partial n_{2}\left(u_{1}\right)}
g_{G_{2}}\left(u_{1},F_{2}^{-1}\left(b\right)\right)/\left|F'\left(u_{1}\right)\right|.
\end{multline*}
\end{prop}
This proposition follows immediately from Proposition 6 from \cite{JMAA} and is a two-to-one mapping analogue of Proposition \ref{prop:green_deriv_unitdisk}.

\begin{proof}[Proof of Theorem \ref{thm:analrationalarc}.]
Use Proposition \ref{prp:openup},
and consider $f_1(u):=f\circ F(u)$ on the analytic Jordan curve $\Gamma$
at $u_1$ (where $F(u_1)=z_0$)
and put $G_1:=F_1^{-1}(\bC_\infty\setminus \Gamma_0)=\mathrm{Int}\Gamma$ and
$G_2:=F_2^{-1}(\bC_\infty\setminus \Gamma_0)=\mathrm{Ext}\Gamma$,
similarly as above.
Obviously, $G=\bCi\setminus \Gamma_0$.
We have then
\begin{multline*}
\left|f_1' \left(u_1\right)\right|
\le 
(1+o(1))
\left\|f_1\right\|_{\Gamma}
\\ \cdot
\max\left(
\sum_\beta
\frac{\partial}{\partial n_1(u_1)}g_{G_1}\left(u_1, F_1^{-1}(\beta)\right),
\sum_\beta
\frac{\partial}{\partial n_2(u_1)}g_{G_2}\left(u_1, F_2^{-1}(\beta)\right)
\right)
\end{multline*}
where $\beta$ runs through the poles of $f(z)$ (counting multiplicities).

Now use Proposition \ref{prp:greenopenup}  and $f_1'\left(u_1\right)=f'\left(z_0\right) F'(u_1)$,
hence
\begin{multline*}
\left|f' \left(z_0\right)\right|
\le 
(1+o(1))
\left\|f\right\|_{\Gamma_0}
\\ \cdot
\max\left(
\sum_\beta
\frac{\partial}{\partial n_1(z_0)}g_{G}\left(z_0, \beta\right),
\sum_\beta
\frac{\partial}{\partial n_2(z_0)}g_{G}\left(z_0, \beta\right)
\right)
\end{multline*}
which we wanted to prove.
\end{proof}

\section{Sharpness - proof of Theorem \ref{thm:analratsharp}}

The idea is as follows. On the unit circle ($v$ plane),
we use some special rational functions
for which  Borwein-Erdélyi inequality is sharp. Then we transfer it back to $\Gamma$
($u$ plane) and approximate it with rational functions. In other words, we do the  ``reconstruction step'' in the  ``opposite direction''.

Recall that $\bD^*=\{z: |z|>1\}\cup\{\infty\}$ and $B(a,v)=\frac{1-\overline{a}v}{v-a}$ is the Blaschke factor with pole at $a$.

First, we state the cases when we have equality in Borwein-Erdélyi inequality.
\begin{prop}
\label{prp:BEsharpness}
Suppose $h$ is a Blaschke product with all poles are either inside or outside the unit circle,
that is,
$h(v)=\prod_{j=1}^n B\left(\alpha_j,v\right)$ where all $\alpha_j\in\bD$,
or $h(v)=\prod_{j=1}^n B\left(\beta_j,v\right)$ where all $\beta_j\in\bD^*$.
Then
\[
\left|h'(1)\right|
=
\left\|h\right\|_{\partial \bD}
\max\left(
\sum_{\alpha}
\frac{\partial}{\partial n_{1}\left(1\right)}
g_{\bD}\left(1,\alpha\right),
\sum_{\beta}
\frac{\partial}{\partial n_{2}\left(1\right)}
g_{\bD^*}\left(1,\beta\right)
\right)
\]
where the sum with $\alpha$ (or $\beta$) is taken over all poles of $h$ in $\bD$ (or in $\bD^*$, respectively), counting multiplicities.
\end{prop}
This proposition is contained in  Borwein-Erdélyi inequality as stated in \cite{BorweinErdelyi} pp. 324-326.

\bigskip

The proof starts as follows. Take $n$ (not necessarily different) points from 
$\Phi_1^{-1}\left(Z_1\right)$, denote them by $\alpha_1,\ldots,\alpha_n$ and let
\[
h_n(v):=\prod_{j=1}^n 
B\left(\alpha_j,v\right).
\]
Obviously $\left\|h_n\right\|_{\partial \bD}=1$.
Now we 
 ``transfer''  $h_n$ to $u$ plane.

Let $f_{1,n}(u)$ be the sum of principal parts of $h_n\left(\Phi_1^{-1}(u)\right)$. 
And approximate $h_n\left(\Phi_1^{-1}(u)\right) - f_{1,n}(u)$ with rational functions with poles outside
$\Gamma$ 
as follows.

It is immediate that  
\[
\varphi_{e}(u):=h_n\left(\Phi_1^{-1}(u)\right) - f_{1,n}(u)
\]
is holomorphic in $G_{1}^+:=\left\{ \Phi_1(v):\,|v|<1+\delta_{1}\right\}$.

Fix any $\zeta_0\in Z_2$ arbitrarily. 
Take a rational function $w=\psi(u)$ such that $\psi\left(\zeta_0\right)=\infty$ and $\deg\psi=1$.
Hence $\psi$ is a conformal mapping from $\bC_\infty$ onto itself.
Let $\Gamma_w:=\psi\left(\Gamma\right)$ and 
$G_{1,w}^+:=\psi\left(G_1^+\right)$ and $w_0:=\psi\left(u_0\right)$.
Fix a smooth Jordan curve $\Gamma_{w+}$ such that $\Gamma_w\subset\mathrm{Int} \Gamma_{w+}$,
$\Gamma_{w+}\subset G_{1,w}^+$ and
$\psi\left(Z_2\right)\cap \left(\Gamma_{w+} \cup \mathrm{Int}\Gamma_{w+} \right)=\emptyset$.
These are depicted on figure \ref{fig:sharpthree}.
\begin{figure}
\begin{center}
\includegraphics[keepaspectratio,width=\textwidth]{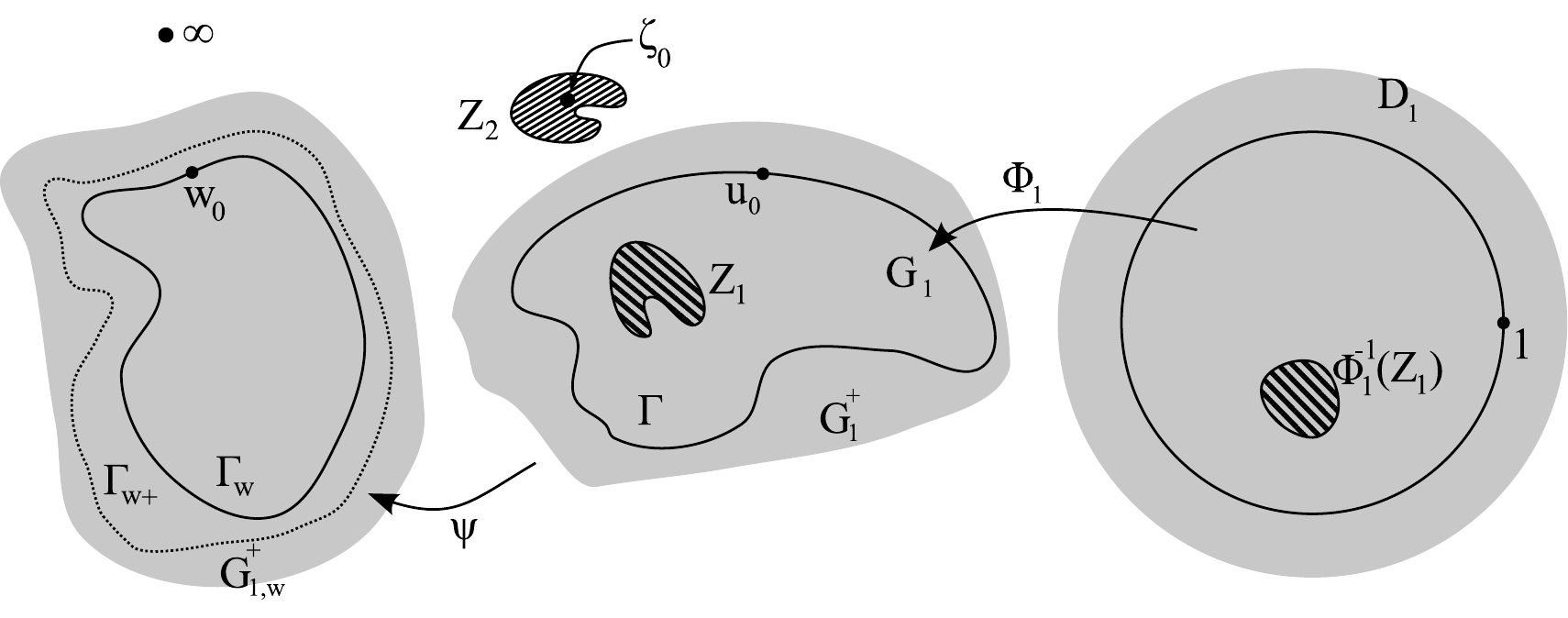}
\end{center}
\caption{The three planes $w$, $u$ and $v$ and
the mappings}
\label{fig:sharpthree}
\end{figure}

Put
\[
N_6:=[n^{4/5}]
\]
and denote a Fekete polynomial (monic polynomial with zeros at a Fekete point set)
of $\Gamma_w$ with degree $N_6$ by $P(w)$.

We are going to use the sharpness of Bernstein-Walsh lemma (see, e.g. \cite{Ransford}, Theorem 5.5.7 (b), p. 156)
therefore we put
\begin{align*}
\cdeltathree &:=
\inf \left\{g_{\bC_\infty\setminus \Gamma_w}\left(w,\infty\right): \  
w\in\Gamma_{w+}\right\} > 0,
\\
\csharpone &:= \sup\left\{\tau(w,\infty):\ w\in\Gamma_{w+}\right\} <\infty
\end{align*}
where $\tau(.,.)$ denotes the Harnack distance.

It is known that the $n$-th diameter $\delta_{N_6}$ of $\Gamma_w$ is close to the capacity $\mathrm{cap}(\Gamma_w)$, (see Fekete-Szegő theorem \cite{Ransford}, Theorem 5.5.2 on page 153), since $\delta_{N_6}\rightarrow \mathrm{cap}(\Gamma_{w})$.
Hence $N_6$ is large, then
$\csharpone \log\frac{\delta_{N_6}}{\mathrm{cap}(\Gamma_w)} < \cdeltathree /2$ and
$g(w,\infty)-\tau(w,\infty) \log \frac{\delta_{N_6}}{\mathrm{cap}(\Gamma_w)} >  \cdeltathree /2$.
This way we have for large $N_6$ that
\begin{equation}
\label{est:Plower}
\left|P(w)\right| \geq
\left\|P\right\|_{\Gamma_w} 
 \exp\left( N_6 \cdeltathree/2\right), 
\quad  (w\in \Gamma_{w+}).
\end{equation}

Let 
\[
q(w):=P(w)
\left(w-w_0\right)^2,
\]
this $q$ is a polynomial with degree $N_6+2$ and 
$w_0$ is (at least) a double zero.
Introduce
\[
f_{2,w}(w):=
\frac{1}{2\pi i}
\int_{\Gamma_{w+}}
\frac{\varphi_{e}\left(\psi^{-1}(w)\right)}{q\left(\xi\right)}
\frac{q\left(w\right)-q\left(\xi\right)}{w-\xi}d\xi
\]
where $w\in \mathrm{Int}\Gamma_{w+}$
and this $f_{2,w}(.)$ is a polynomial with degree $N_6+2$.
Here, $f_{2,w}'\left(w_0\right)=\varphi_e'\left(u_0\right) / \psi'\left(u_0\right)$.

The error of interpolation on the $w$ plane is
\begin{equation*}
\varphi_{e}\left(\psi^{-1}(w)\right)-f_{2,w}\left(w\right) 
=  \frac{1}{2\pi i}
\int_{\Gamma_{w+}}
\frac{\varphi_{e}\left(\psi^{-1}(\xi)\right)}{\xi-w}
\frac{q\left(w\right)}{q\left(\xi\right)}d\xi,
\end{equation*}
where $w\in \mathrm{Int}\Gamma_{w+}$.
It can be estimated from above as follows if $w\in \Gamma_w$
\begin{multline}
\left|\varphi_{e}\left(\psi^{-1}(w)\right)-f_{2,w}\left(w\right) \right|
\le \frac{1}{2\pi}
\int_{\Gamma_{w+}}
\left|\frac{\varphi_{e}\left(\psi^{-1}(\xi)\right)}{\xi-w}\right|
\left|\frac{q\left(w\right)}{q\left(\xi\right)}\right|
\left|d\xi\right|
\\ \le
\frac{1}{2\pi}
\int_{\Gamma_{w+}}
\left|\varphi_{e}\left(\psi^{-1}(\xi)\right)\right|
\left|\frac{P\left(w\right)}{P\left(\xi\right)}\right| \csharpthree
\frac{1}{|\xi-w|\left|\xi-w_0\right|^2} |d\xi|
\label{est:sharpf2error}
\end{multline}
where $\csharpthree>0$ (actually $\csharpthree=\mathrm{diam}\left(\Gamma_{w+}\right)^2$)
and we continue this estimate later.
Here $\left|\varphi_e\left(\psi^{-1}(\xi)\right)\right|$, $\xi\in\Gamma_{w+}$
can be estimated using that $h_n(v)$ is a Blaschke product with all poles in the unit disk,
therefore $\left\|h_n\circ\Psi_1^{-1}\circ \psi^{-1}\right\|_{\Gamma_{w+}} \le 1$ and using the Gonchar-Grigorjan type 
estimate (see Theorem 1 in \cite{multiplygg}) on 
$\mathrm{Int} \Gamma$ for $h_n\circ \Phi_1^{-1}$ which implies that 
$\left\|f_{1,n}\right\|_{\Gamma}
\le \constgg(\Gamma) 
\log(n) \left\| h_n\circ\Phi_1^{-1}\right\|_\Gamma
= \constgg(\Gamma) \log(n)$
and, since $f_{1,n}$ has poles in $\mathrm{Int}\Gamma$, the maximum principle yields $\left\|f_{1,n}\right\|_{\Gamma_w}
\ge\left\|f_{1,n}\right\|_{\Gamma_{w+}}$.
Hence 
\begin{multline}
\left|\varphi_e\left(\psi^{-1}(\xi)\right)\right|
\le
\left\|h_n\circ\Psi_1^{-1}\circ \psi^{-1}\right\|_{\Gamma_{w+}}
+
\left\|f_{1,n}\right\|_{\Gamma_{w+}}
\\ \le
1+\constgg(\Gamma) \log(n)=O(\log(n))
\label{est:phieongammawp}
\end{multline}
where $n$ is large enough (depending on $\Gamma$ only).
We also use
\[
\frac{\csharpthree}{2\pi}
\int_{\Gamma_{w+}}
\frac{1}{|\xi-w| \left|\xi-w_0\right|^2  } |d\xi|
\le \csharptwo
\] 
for some constant $\csharptwo>0$
depending on $\Gamma$ and $\Gamma_{w+}$,
since $w\in\Gamma_w$, and $\Gamma_{w+}$ is fixed and is from positive distance from  $\Gamma_w$.
Finally, we estimate $|q(w)/q(\xi)|$ using \eqref{est:Plower} (when $w\in \Gamma_w$, $\xi\in\Gamma_{w+}$),
whence
\[
\left|\frac{q(w)}{q(\xi)}\right|
\le
\frac{\left\|P\right\|_{\Gamma_w}}
{\left\|P\right\|_{\Gamma_w} 
 e^{N_6 \cdeltathree/2}}
=  e^{-N_6 \cdeltathree /2}.
\]
Putting these estimates together,
we can finish \eqref{est:sharpf2error}
\begin{equation*}
\le O\left( \log(n) 
e^{-N_6 \cdeltathree /2}\right).
\end{equation*}
Now substituting $w=\psi(u)$,
we can write
\begin{equation}
\left\|\varphi_e-f_{2,w}\circ\psi\right\|_{\Gamma}
\le  O\left( \log(n) 
e^{-N_6 \cdeltathree /2}\right)
\label{est:f2ongamma}
\end{equation}
where $\delta_3$ and the constant in $O(.)$ is independent of $h_n$
and $n$ and depends only on $\Gamma$ and $\Gamma_{w+}$ only.

\medskip

Finally, let 
\[
f_n(u):=f_{1,n}(u)+f_{2,w}\left(\psi(u)\right)
\]
this is a rational function with poles in $\mathrm{Int}\Gamma$ with total order $n$ and one pole in the exterior of $\Gamma$ (at $\zeta_0$) with order $N_6+2=O(n^{4/5})=o(n)$.
Moreover
\[
\left|f_n'\left(u_0\right)\right|
= 
\left|f_{1,n}'\left(u_0\right)
+ \left(\varphi_e'\left(u_0\right)/\psi'\left(u_0\right) \right) \psi'\left(u_0\right)
\right|=
\left|h_n'(1)\right|
\]
since $\left|\Phi_1'(1)\right|=1$.
We know that $\left\|h_n\right\|_{\partial\bD}=1$ and
using \eqref{est:f2ongamma},
\begin{multline*}
\left\|f_n\right\|_\Gamma
= \left\| f_{1,n} + f_{2,n}\circ\psi \right\|_\Gamma
=\left\| f_{1,n} + \varphi_e + f_{2,n}\circ\psi -\varphi_e \right\|_\Gamma
\\= 
\left\| h_n\circ\Phi_1^{-1} + f_{2,n}\circ\psi -\varphi_e \right\|_\Gamma
=1\pm O\left( \log(n) 
e^{-N_6 \cdeltathree /2}\right)
\end{multline*}
therefore $\left\|f_n\right\|_\Gamma\le (1+o(1)) \left\|h_n\right\|_{\partial\bD}$.

We use Propostition  \ref{prp:BEsharpness} for $h_n$, hence
\begin{multline*}
\left|f_n'\left(u_0\right)\right|
=
\left|h_n'(1)\right|
\\ =
\left\|h_n\right\|_{\partial \bD}
\max\left(
\sum_{\alpha}
\frac{\partial}{\partial n_{1}\left(1\right)}
g_{\bD}\left(1,\alpha\right),
\sum_{\beta}
\frac{\partial}{\partial n_{2}\left(1\right)}
g_{\bD^*}\left(1,\beta\right)
\right)
\\ \ge 
(1-o(1))\left\|f_n\right\|_\Gamma
\max\left(
\sum_{\alpha}
\frac{\partial}{\partial n_{1}\left(1\right)}
g_{\bD}\left(1,\alpha\right),
\sum_{\beta}
\frac{\partial}{\partial n_{2}\left(1\right)}
g_{\bD^*}\left(1,\beta\right)
\right)
\end{multline*}
where actually the second term in the max is void ($h_n$ has poles only inside the unit disk) and
by Proposition \ref{prop:green_deriv_unitdisk},
so
\begin{multline*}
\max\left(
\sum_{\alpha}
\frac{\partial}{\partial n_{1}\left(1\right)}
g_{\bD}\left(1,\alpha\right),
\sum_{\beta}
\frac{\partial}{\partial n_{2}\left(1\right)}
g_{\bD^*}\left(1,\beta\right)
\right)
=
\sum_{\alpha}
\frac{\partial}{\partial n_{1}\left(1\right)}
g_{\bD}\left(1,\alpha\right)
\\ =
\sum_{\alpha}
\frac{\partial}{\partial n_{1}\left(u_0\right)}
g_{G_1}\left(u_0,\Phi_1(\alpha)\right)
\\=
\max\left(
\sum_{\alpha}
\frac{\partial}{\partial n_{1}\left(u_0\right)}
g_{G_1}\left(u_0,\Phi_1(\alpha)\right),
\left( N_6+2 \right)
\frac{\partial}{\partial n_{2}\left(u_0\right)}
g_{G_2}\left(u_0,\zeta_0\right)
\right)
\end{multline*}
in the last step we used that the first term in the max is $O(n)$ and the second one is $o(n)$, hence if $n$ is large depending on $\Gamma$, then the last equality holds.

This finishes the proof of Theorem \ref{thm:analratsharp}.

\section{Acknowledgement}

This research was partially done 
when the first author had a postdoctoral position at the Bolyai
Institute, University of Szeged, Aradi v. tere 1, Szeged 6720, Hungary and was supported by the European Research Council Advanced
grant No. 267055.
He was partially supported by the Russian Science Foundation under grant 14-11-00022 too (Theorem \ref{thm:analratsharp}).

The authors would like to thank Vilmos Totik for the interest and
useful discussions.


\providecommand{\bysame}{\leavevmode\hbox to3em{\hrulefill}\thinspace}
\providecommand{\MR}{\relax\ifhmode\unskip\space\fi MR }
\providecommand{\MRhref}[2]{%
  \href{http://www.ams.org/mathscinet-getitem?mr=#1}{#2}
}
\providecommand{\href}[2]{#2}

\bigskip

Sergei Kalmykov
\\
%
Department of Mathematics, Shanghai Jiao Tong University,
800 Dongchuan RD, Shanghai, 200240, China 
\\
and
\\
Far Eastern Federal University, 8 Sukhanova Street, Vladivostok, 690950,
Russia 
\\
email address: \href{mailto:sergeykalmykov@inbox.ru}{sergeykalmykov@inbox.ru}

\medskip{}

Béla Nagy
\\
MTA-SZTE Analysis and Stochastics Research Group, Bolyai Institute,
University of Szeged, Szeged, Aradi v. tere 1, 6720, Hungary
\\
email address: \href{mailto:nbela@math.u-szeged.hu}{nbela@math.u-szeged.hu}

\end{document}